\documentclass[a4paper,11pt]{article}
\usepackage[utf8]{inputenc}
\usepackage[margin=2.5cm]{geometry}
\usepackage{amssymb}
\usepackage{latexsym}
\usepackage{amsfonts}
\usepackage{amsthm}
\usepackage{amsmath}
\usepackage{tikz}

\usepackage[hidelinks]{hyperref}
\usetikzlibrary{arrows.meta,
                chains,
                positioning,
                shapes.geometric
                }
\usepackage{tcolorbox}
\usepackage{float}
\usepackage{bbm}
\usepackage{enumerate}
\usetikzlibrary{arrows.meta}
\usepackage[normalem]{ulem}
\usepackage{mathtools}
\usepackage{authblk}

\usepackage{algorithm}
\usepackage{algorithmicx}
\usepackage{algpseudocode}

\usepackage{wrapfig,lipsum}
\usepackage{caption}
\usepackage{subcaption}
\usepackage[shortlabels]{enumitem}
\usepackage{tasks}
\newtheorem{theorem}{Theorem}[section]
\newtheorem{lemma}[theorem]{Lemma}
\newtheorem{prop}[theorem]{Proposition} 
\newtheorem{cor}[theorem]{Corollary}
  
\theoremstyle{definition}
\newtheorem{defn}[theorem]{Definition}

\newcommand{\Rel}{\textrm{Rel}} 
\newcommand{\Ord}{\textrm{Ord}} 
\newcommand{\Ess}{\textrm{Ess}} 
\newcommand{\DC}{\text{DC}} 
\newcommand{\ext}{\text{ext}}  
\newcommand{\birth}{\mathfrak{b}}  
\newcommand{\death}{\mathfrak{d}}  
\newcommand{\R}{\mathbb{R}}
\newcommand{\im}{\text{im}}
\newcommand{\Z}{\mathbb{Z}}

\newcommand{\argmax}{\text{argmax}}
\newcommand{\Cascade}{\text{Cascade}}

\title{Computing extended persistent homology of radial distance filtrations of {E}uclidean shapes} 
\author[1]{Yuchen Jency Jiang} 
\author[2]{Vanessa Robins}
\author[3]{Katharine Turner}
\affil[1]{Division of Applied Mathematics, Brown University}
\affil[2]{Research School of Physics, Australian National University}
\affil[3]{Mathematical Sciences Institute, Australian National University}
\date{\today}

\begin{document}

\maketitle

\begin{abstract}
   We study the extended persistent homology of the radial filtration of a shape $M\subseteq \mathbb{R}^n$. A radial filtration is formed by choosing a center point $c$ and taking points of $M$ within distance $r$ of $c$. We show that, under mild assumptions, we can recover the \emph{radial extended persistence} of a manifold with boundary from the radial persistence of its boundary. We also establish algorithms to compute radial extended persistence when $M$ is a set of pixels in a 2D digital grid. The methods are similar to those used to compute the extended persistent homology of a height filtration of a shape embedded in Euclidean space. We envisage these results will be useful in biomedical image analysis settings where it is natural to consider a radial filtration with respect to a fixed center, for example in the study of neuronal structures.
\end{abstract}

\section{Introduction}

Topological data analysis (TDA) provides methods to rigorously quantify the shape of data and to measure distances between two shape summaries. 
These methods are grounded in algebraic topology and can be adapted to a wide range of contexts and applications. 
In this work, we focus on a context relevant to images of biological specimens such as leaf shapes, bone anatomy or biometrics and will measure these specimens using filtrations of \emph{radial distance functions}. 

Mathematically, each specimen $M$ is an $n$-dimensional subset of $\R^n$ and its shape is quantified using the extended persistent homology of a filtration of $M$. 
We use extended persistence to overcome one of the drawbacks of `regular' persistent homology: that the distance between two shapes is infinite unless they have the same Betti numbers (e.g., the same number of components or holes). Extended persistence was introduced in~\cite{dualityMore} to address this problem. 

The simplest possible filtration of a shape is one built from sublevel sets of a height function. 
In~\cite{KT2022}, it was shown that when $M$ is an $n$-manifold with boundary in $\R^n$, the extended persistent homology of a height function on $M$ can be efficiently recovered from its restriction to the boundary of $M$.
The current paper establishes the same result for a different class of functions: radial distance from a point. 
Specifically, given a point $c \in \R^n$, the radial function on $M$ is $\rho_c(x) = \|x-c\|_2$, the Euclidean distance between $x \in M$  and $c$. 

For image data where a radial filtration is desirable, there is typically a natural choice of center point that can be automatically determined (such as the nucleus of a cell or the centroid of a region). 
A radial distance function then has the property that it is insensitive to orientation of the image. 
This is desirable when comparing many images, as otherwise an extra step of image alignment is required. 
Radial filtrations have found applications in biomedical diagnosis \cite{A2020,D2024,J2025}, dynamic network analysis \cite{S2022}, as well as in machine learning tasks such as image classification, segmentation and reconstruction \cite{C2022,G2019,N2024,T2021}. Note that those implementations all used a persistent homology pipeline instead of the extended persistence counterpart studied here. 

The paper proceeds as follows. In Section~\ref{sec:bg} we establish the definitions and notation for radial extended persistence modules. The technical results rely on Morse theory for smooth and piecewise-linear manifolds with boundary, and the required material is covered in Section~\ref{sec:morseTheory}. 
The difference between the previous results for height functions~\cite{KT2022}, and the current analogous ones for radial distance functions is that the relationship between the persistence modules for $M$ and $\partial M$ depends on the location of the center point $c$ with respect to $M$. The different cases are explained and explored in full through a series of lemmas and theorems in Section~\ref{sec:proof}. 
We show how to adapt these results to establish algorithms for binary images in Section~\ref{application}. 
As in~\cite{KT2022}, the correctness of the algorithms relies on a specific method for constructing the boundary curve around the foreground object, defined in~\ref{consider}. The discrete uniform grid of a digital image also means the numerical version of a radial distance function can easily fail to satisfy the assumptions of the theorems in Section~\ref{sec:proof}. We explain the issues and workarounds in Sections~\ref{signs} and~\ref{practical}.
Pseudocode for algorithms that cover all cases for the location of $c$ is provided in the Appendix (Section \ref{code}) and an implementation in R is available for the simpler cases.

\section{The radial extended persistent homology module}
\label{sec:bg}

We assume that readers are familiar with persistence modules, their interval decompositions and relative homology. 
While we briefly describe the construction of an extended persistence module here, we recommend those unfamiliar with extended persistence to first read the substantial exposition in \cite{KT2022}. We use the same notation as in that paper.  

Intuitively, we think of extended persistent homology as first scanning upwards (of some sublevel set filtration) gradually revealing more information, then sweeping back down using relative homology until we get back to the starting point.  Since the beginning and the end of the sequence are both trivial, all components that are born along the way will eventually die. 
To make it clear that extended persistence is a natural extension of regular persistence, we define persistence modules over a general parameter space, namely a \emph{totally ordered set} \cite{KT2022}. 
For example, in regular persistent homology the parameter space is typically chosen to be a subset of $\R$ endowed with the usual order.  
Our extended persistence modules are equipped with an ``extended" parameter space $\Theta$ consisting of two copies of $\R$ as follows. Let $\mathcal{O} = \{(t,\text{Ord}): t\in\R\}$ and $\mathcal{R}=\{(t,\text{Rel}): t\in\R\}$. Then $\Theta=\mathcal{O}\cup \mathcal{R}$ and equipped with the total order:
\begin{itemize}
    \item $(s,\text{Ord}) \leq (t,\text{Ord})$ if $s\leq t$;
    \item $(t,\text{Rel}) \leq (s,\text{Rel})$ if $s\leq t$;
    \item $(s,\text{Ord})<(t,\text{Rel})$ for all $s,t\in\R$.
\end{itemize}

Given a finite geometric simplicial complex $M$ embedded in $\R^n$ and a smooth function $f: M \to \R$, the topological space associated with $(t,\text{Ord})$ is  the sublevel set $M_t = f^{-1}(-\infty,t]$, while the space associated with $(t,\text{Rel})$ is the superlevel set $M^t = f^{-1}[t,\infty)$. 
Ordinary persistent homology maps the sequence of sublevel sets to vector spaces $H_k(M_t)$ (homology groups with $\Z_2$ coefficients). 
Upgrading this to extended persistence, 
to each parameter in $\Theta$ we attach a vector space $V_{(t,\text{Ord})} = H_k(M_t)= H_k(M_t, \emptyset)$ and $V_{(t,\text{Rel})} = H_k(M,M^t)$, the relative homology of a pair. 
The maps between the vector spaces $V_\theta$ are those induced by inclusion maps between corresponding pairs of spaces and the interval decomposition theorem for persistence modules applies~\cite{crawley2015decomposition}.
Intervals that have both endpoints in $\mathcal{O}$ are \emph{ordinary classes} and those with both endpoints in $\mathcal{R}$ are \emph{relative homology classes}. 
The \emph{essential homology classes} of $M$ are those with a birth at $(s,\text{Ord}) \in \mathcal{O}$, and these now have a death at some $(t,\text{Rel}) \in \mathcal{R}$, with $s,t \in \R$.  
We further classify the essential classes into two subcategories, $\text{Ess}^+$ and $\text{Ess}^-$, where the former has $s<t$ and the latter satisfies $t<s$. Then we have that the degree-$k$ extended persistence module $\mathrm{XPH}_k(M,f)$ induced by $f$ decomposes as
\begin{equation*}
\mathrm{XPH}_k(M,f) = \text{Ord}_k(M,f)\oplus \text{Rel}_k(M,f)\oplus \text{Ess}_k^+(M,f)\oplus \text{Ess}_k^-(M,f), 
\end{equation*}
where each summand is a submodule over a copy of $\R$ corresponding to the given class \cite{KT2022}.
Note that the essential classes are also referred to as extended classes in earlier literature. 

\bigskip 
Fix a centre $c\in \R^n$, we define the \emph{radial distance function} centred at $c$:  
\begin{align*}\rho_c: M&\to \R_{\ge 0} \\
x&\mapsto \lVert x-c\rVert_2 = \sqrt{(x-c)\cdot (x-c)},
\end{align*}
where $\cdot$ is the inner product inherited from $\R^n$. 
We use $\rho_c$ to define an extended filtration of $M$, using the sub- and superlevel sets at threshold $r$: 
\begin{align*} 
    M_r &\coloneqq \left\{x\in M:\rho_c(x) \leq r\right\} \\ 
    M^r  &\coloneqq \left\{x\in M:\rho_c(x) \geq r\right\}   
\end{align*}
The \emph{degree-$k$ extended radial persistence module} is then defined to be 
$\text{XRPH}_k(M, c)\coloneqq \text{XPH}_k(M, \rho_c)$, and the full 
\emph{extended radial persistent homology (XRPH) of $M$ with centre $c$} as 
$$\mathrm{XRPH}(M, c) = \big(\mathrm{XRPH}_0(M,c),\mathrm{XRPH}_1(M,c),\ldots,\mathrm{XRPH}_n(M,c)\big).$$

\section{Signs of critical points and Morse functions}
\label{sec:morseTheory}
 
The main results of this paper use Morse theory, which relates the topology of a manifold to the critical points of differentiable functions on the manifold. As in~\cite{KT2022} both the smooth and piecewise linear (PL) versions of Morse theory are needed: the former for the theoretical development, and the latter for establishing the validity of numerical algorithms applied to PL manifolds as described in Section~\ref{application}. 
Note that the PL case here is not discrete Morse theory, but refers to functions that are defined on the vertices of a geometric simplicial complex and interpolated linearly on the simplices. 
In this section, we follow the setup described in \cite{KT2022} and briefly summarise the necessary points.  

\medskip

First, we need to define what it means for a point on a manifold to be a critical point. The standard way to define critical points is in terms of taking derivatives with respect to some charts. As noted in \cite{KT2022}, we can instead adopt an equivalent definition justified using the Morse Lemma.
\begin{defn}[Critical points of smooth functions]
Let $M$ be a smooth $n$-manifold \emph{without boundary} and $f: M\to \R$ be a smooth function. 
Then $p\in M$ is a \emph{regular point} of $f$ if there is a chart $(U,\phi)$ such that $\phi(p)=0$ and for $x=(x_1,\cdots,x_n)\in\R^n$ in a neighbourhood of $0$, we have
$$f\circ\phi^{-1}(x) = f(p)+x_n.$$ 
We call $p\in M$ a \emph{non-degenerate critical point of $f$ with Morse index $k$} if there is a chart $(U,\phi)$ such that $\phi(p)=0$ and for $x=(x_1,\cdots,x_n)\in\R^n$ in a neighbourhood of $0$, we have
$$f\circ\phi^{-1}(x) = f(p)-x_1^2-x_2^2-\cdots -x_k^2+x_{k+1}^2+\cdots + x_n^2.$$
\end{defn}

A similar definition for PL functions can be made using absolute values.
\begin{defn}[Critical points of piecewise linear functions]
Let $M$ be a PL $n$-manifold without boundary and $f: M\to \R$ a PL function. Then $p\in M$ is a \emph{regular point} of $f$ if there is a chart $(U,\phi)$ such that $\phi(p)=0$ and for $x=(x_1,\cdots,x_n)\in\R^n$ in a neighbourhood of $0$, we have
$$f\circ\phi^{-1}(x) = f(p)+x_n.$$ 
We call $p\in M$ a \emph{non-degenerate critical point of $f$ with Morse index $k$} if there is a chart $(U,\phi)$ such that $\phi(p)=0$ and for $x=(x_1,\cdots,x_n)\in\R^n$ in a neighbourhood of $0$, we have
$$f\circ\phi^{-1}(x) = f(p)-\lvert x_1\rvert-\lvert x_2\rvert-\cdots -\lvert x_k\rvert+\lvert x_{k+1}\rvert+\cdots + \lvert x_n\rvert.$$
\end{defn}

Since we model shapes as manifolds with boundary, we define critical points on the boundary using a similar formulation. These boundary critical points have an additional signature as  $(+)$- and $(-)$-critical points (compare with Definition \ref{localmin}). 
Note that critical points in the interior are treated as in the above definitions.  
 
\begin{defn}[Boundary critical points of smooth functions]
Let $(M,\partial M)$ be a smooth $n$-manifold with boundary and $f: M\to \R$ a smooth function. 
Then $p\in \partial M$ is a \emph{non-degenerate critical point of $f$ with index $(k,\eta)$} if there exists a chart $(U,\phi)$ such that $\phi(p) = 0$ and for $x= (x_1,\cdots,x_n)$ in a neigbourhood of $0$ in the closed half-space $\{(x_1,\cdots,x_n)\in\R^n\mid x_1\geq 0\}$ of $\R^n$, we have $\phi\vert_{\partial M}$ has the first coordinate $x_1=0$ and
$$f\circ \phi^{-1}(x) = f(p)+\eta x_1 - x_2^2-\cdots-x_{k+1}^2+x_{k+2}^2+\cdots x_n^2,$$ where $\eta\in \{-1,1\}$. We call $p$ a \emph{$(+)$-critical point} if $\eta=1$ and \emph{$(-)$-critical point} if $\eta=-1$.
\end{defn}

The analogous definition in the piecewise linear case is as expected: 
\begin{defn}[Boundary critical points of piecewise linear functions]
Let $(M,\partial M)$ be a piecewise linear $n$-manifold with boundary and $f: M\to \R$ a piecewise function. Then $p\in \partial M$ is a \emph{non-degenerate critical point of $f$ with index $(k,\eta)$} if there exists a chart $(U,\phi)$ such that $\phi(p) = 0$ and for $x= (x_1,\cdots,x_n)$ in a neigbourhood of $0$ in the closed half-space $\{(x_1,\cdots,x_n)\in\R^n\vert x_1\geq 0\}$ of $\R^n$, we have
$\phi\vert_{\partial M}$ has the first coordinate $x_1=0$ and 
$$f\circ \phi^{-1}(x) = f(p)+\eta x_1 - \lvert x_2\rvert-\cdots- \lvert x_{k+1}\rvert+ \lvert x_{k+2}\rvert+\cdots  \lvert x_n\rvert,$$ where $\eta\in \{-1,1\}$. We call $p$ a \emph{$(+)$-critical point} if $\eta=1$ and \emph{$(-)$-critical point} if $\eta=-1$. 
\end{defn}

We are now ready to formally define Morse functions. 

\begin{defn}[Morse function]
Let $(M, \partial M)$ be a smooth (or PL) manifold. A function $f: M\to\R$ is a \emph{Morse function} if
\begin{enumerate}[(i)]
\item $f$ is smooth (or piecewise linear);
\item Each critical point of $f\vert_{\text{int}(M)}$ and $f\vert_{\partial M}$ is non-degenerate;
\item The combined number of critical points is finite and they all take distinct values.
\end{enumerate}
\end{defn}

It is easy to see from the definition of $\rho_c$ that radial functions defined on a finite simplicial complex for almost all choices of centre point $c$ will satisfy the requirements to be a Morse function. In particular, a sufficient condition is that all vertices have distinct distances to $c$. We call a centre $c$ \emph{generic} if $\rho_c: M\to \R$ is Morse.
\section[XRPH of manifolds with boundary]{Relating XRPH of a manifold to XRPH of its boundary}
\label{sec:proof}

In this section, we show that we can recover the XRPH of a compact $n$-manifold with boundary $(M,\partial M)$ from the XRPH of the boundary components, $\partial M$. From now on, the manifold of interest will always be denoted with $(M,\partial M)$. Note that if $M=\sqcup_{i=1}^m M_i$ is a disjoint union of connected components, we have $$\text{XRPH}(M,c)=\bigoplus_{i=1}^m\text{XRPH}(M_i,c).$$ 
So we assume without loss of generality in the following that $M$ is connected.

The relationship between the two modules depends on the location of the centre $c$ with respect to $M$ and $\partial M$. In particular, we show that the degree-0 and degree-$n$ extended radial persistence modules depend on whether $c\in \mathrm{int}\,M$ or $c\not\in M$. 
For example, if $c\in \mathrm{int}\, M$ then $\mathrm{XRPH_0}(M,c)$ has a component born at $(0,\mathrm{Ord})$, but no such component exists in $\mathrm{XRPH}_0(M,c)$ when $c\not\in M$ and 
there is a non-trivial interval module in $\mathrm{XRPH}_n(M,c)$ only when $c\in \mathrm{int}\, M$. 
The degree-$(n-1)$ extended radial persistence module, on the other hand, depends on whether the minimum value of $\rho_c$ restricted to $\partial M$ is achieved on what we call the \textit{exterior boundary component} of $M$. If this is the case, the matching between the XRPH of the manifold and its boundary can be straightforward. Otherwise, a ``cascading" behavior described in \cite{matching} occurs, and an alternative algorithm is needed.

Since the above two considerations regarding the location of $c$ impact extended radial persistence modules of different homology degrees, we can decouple their effects so that two examples are sufficient to demonstrate them all. Figures \ref{fig:four-panel} and \ref{fig:four-panel2} illustrate all of the above cases for $n=2$. In Figure \ref{fig:four-panel}, we have the case $c\notin M$ and $\rho_c\vert_{\partial M}$ achieves its minimum value $r_1$ on the exterior boundary.  We see that all interval modules in $\mathrm{XRPH}_k(M,c)$ occur in $\mathrm{XRPH}_k(\partial M,c)$ for $k=0,1$.  
Meanwhile, in Figure \ref{fig:four-panel2}, we have $c\in M$ and $\rho_c\vert_{\partial M}$ achieves its minimum value $r_1$ on the interior boundary. Note also that the interval decompositions in $\mathrm{XRPH}_0(M,c)$ do not have all the birth parameters that occurred in $\mathrm{XRPH}_0(\partial M,c)$. Moreover, the interval module $\mathrm{XRPH}_1(M,c)$ has a birth parameter matching one of the interval modules in $\mathrm{XRPH}_1(\partial M,c)$ but death parameter matching that of another, signaling the cascading behavior. The captions of the figures give a more algorithmic description of the matching.

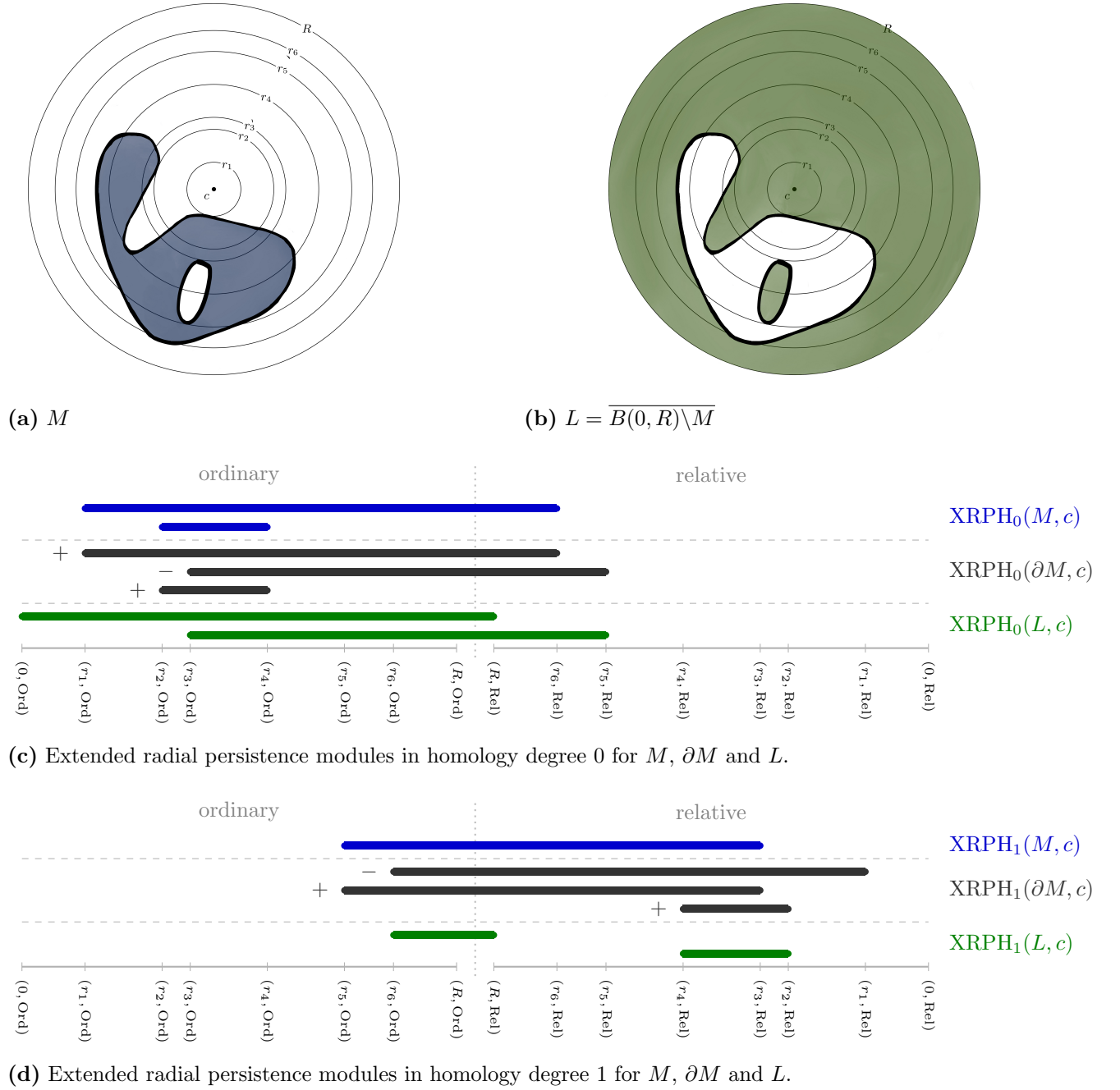
\begin{figure}[htbp]
    \begin{center}
    \begin{tikzpicture}
        \node[anchor=north west] at (0, 0) {%
            \includegraphics[width=0.4 \textwidth]{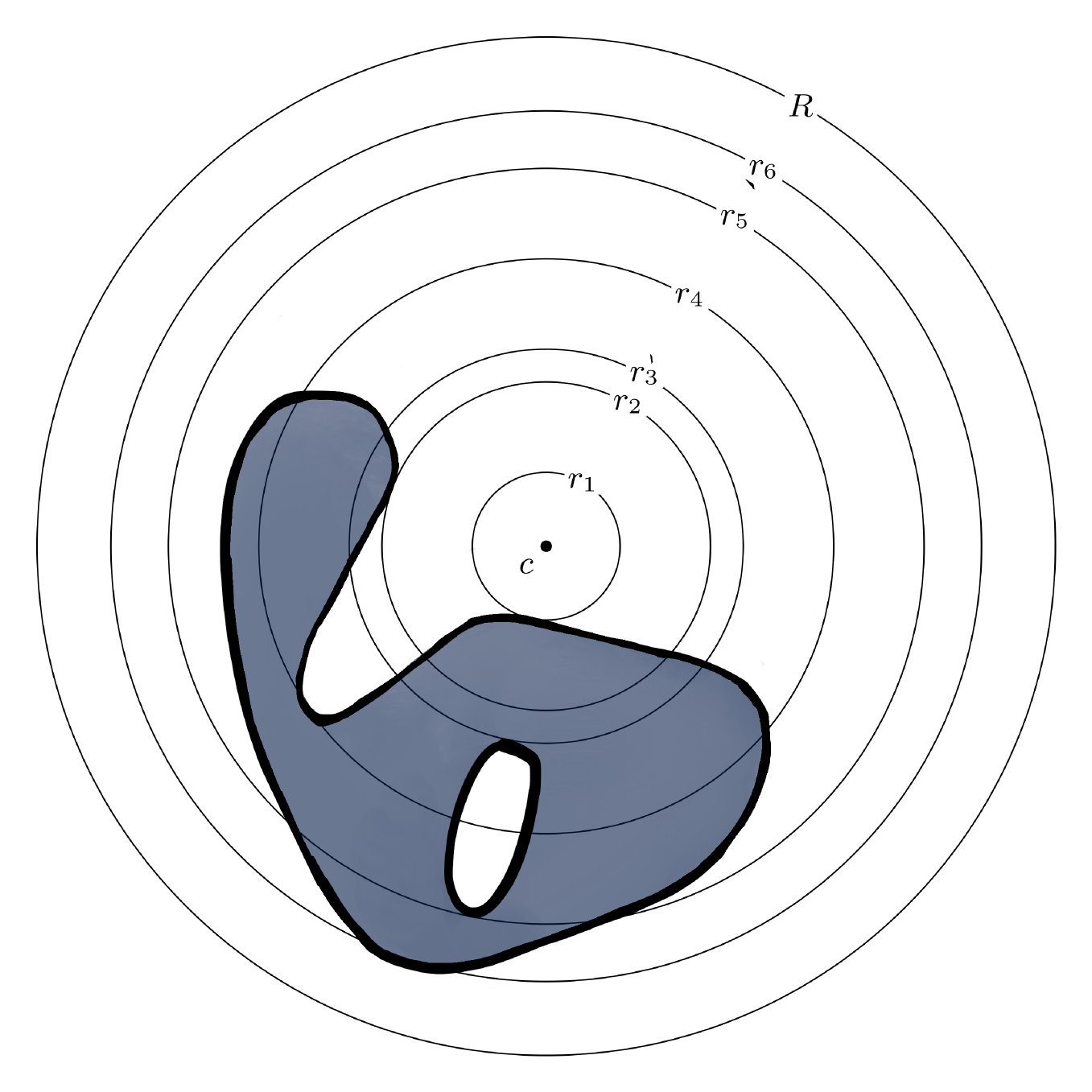}

        };
    \end{tikzpicture}%
    \hfill
    \begin{tikzpicture}
        \node[anchor=north west] at (0, 0) {%
            \includegraphics[width=0.4 \textwidth]{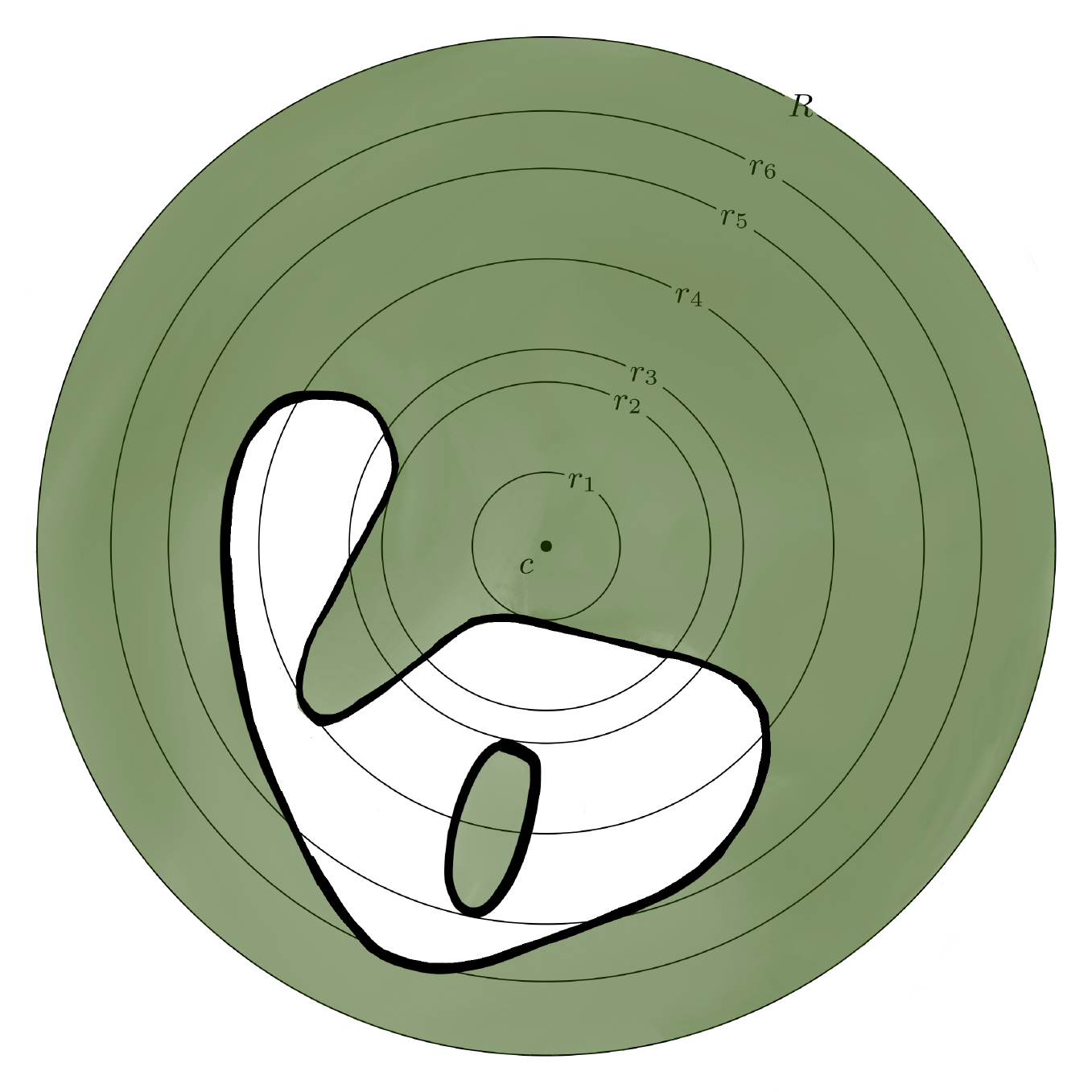}  
        };
    \end{tikzpicture}
    \makebox[0.48\textwidth][l]{\small\textbf{(a)} $M$}%
    \hfill
    \makebox[0.48\textwidth][l]{\small\textbf{(b)} $L=\overline{B(0,R)\backslash M}$}

    \vspace{1em}
\begin{tikzpicture}[x=1cm, y=1cm]

\def\tickA{0.000}   
\def\tickB{1.016}   
\def\tickC{2.258}   
\def\tickD{2.710}   
\def\tickE{3.952}   
\def\tickF{5.194}   
\def\tickG{5.984}   
\def\tickR{7.000}   
\def\tickRr{7.600}  
\def\tickGr{8.616}  
\def\tickFr{9.406}  
\def\tickEr{10.648} 
\def\tickDr{11.890} 
\def\tickCr{12.342} 
\def\tickBr{13.584} 
\def\tickAr{14.600} 

\def\barH{0.12}
\pgfmathsetmacro{\axisline}{0.7 - 0.15}  
\draw[gray!60, line width=0.8pt]
  (\tickA,  \axisline) -- (\tickR,  \axisline);
\draw[gray!60, line width=0.8pt]
  (\tickRr, \axisline) -- (\tickAr, \axisline);
\draw[gray!50, dotted, line width=0.8pt]
  (7.3, 0.40) -- (7.3, 3.5);
\node[font=\small, text=gray] at (3.5,  3.35) {ordinary};
\node[font=\small, text=gray] at (11.1, 3.35) {relative};
\foreach \x in {
  \tickA, \tickB, \tickC, \tickD, \tickE, \tickF, \tickG, \tickR,
  \tickRr, \tickGr, \tickFr, \tickEr, \tickDr, \tickCr, \tickBr, \tickAr
}{
  \draw[gray!60] (\x, \axisline-0.07) -- (\x, \axisline+0.07);
}
\node[font=\tiny, rotate=270, anchor=west] at (\tickA,  \axisline-0.07) {$(0,\mathrm{Ord})$};
\node[font=\tiny, rotate=270, anchor=west] at (\tickB,  \axisline-0.07) {$(r_1,\mathrm{Ord})$};
\node[font=\tiny, rotate=270, anchor=west] at (\tickC,  \axisline-0.07) {$(r_2,\mathrm{Ord})$};
\node[font=\tiny, rotate=270, anchor=west] at (\tickD,  \axisline-0.07) {$(r_3,\mathrm{Ord})$};
\node[font=\tiny, rotate=270, anchor=west] at (\tickE,  \axisline-0.07) {$(r_4,\mathrm{Ord})$};
\node[font=\tiny, rotate=270, anchor=west] at (\tickF,  \axisline-0.07) {$(r_5,\mathrm{Ord})$};
\node[font=\tiny, rotate=270, anchor=west] at (\tickG,  \axisline-0.07) {$(r_6,\mathrm{Ord})$};
\node[font=\tiny, rotate=270, anchor=west] at (\tickR,  \axisline-0.07) {$(R,\mathrm{Ord})$};
\node[font=\tiny, rotate=270, anchor=west] at (\tickRr, \axisline-0.07) {$(R,\mathrm{Rel})$};
\node[font=\tiny, rotate=270, anchor=west] at (\tickGr, \axisline-0.07) {$(r_6,\mathrm{Rel})$};
\node[font=\tiny, rotate=270, anchor=west] at (\tickFr, \axisline-0.07) {$(r_5,\mathrm{Rel})$};
\node[font=\tiny, rotate=270, anchor=west] at (\tickEr, \axisline-0.07) {$(r_4,\mathrm{Rel})$};
\node[font=\tiny, rotate=270, anchor=west] at (\tickDr, \axisline-0.07) {$(r_3,\mathrm{Rel})$};
\node[font=\tiny, rotate=270, anchor=west] at (\tickCr, \axisline-0.07) {$(r_2,\mathrm{Rel})$};
\node[font=\tiny, rotate=270, anchor=west] at (\tickBr, \axisline-0.07) {$(r_1,\mathrm{Rel})$};
\node[font=\tiny, rotate=270, anchor=west] at (\tickAr, \axisline-0.07) {$(0,\mathrm{Rel})$};
\fill[blue!80!black]
  (\tickB, 2.74) rectangle (\tickGr, 2.74+\barH);
\fill[blue!80!black] (\tickB,  2.74+\barH/2) circle (1.5pt);
\fill[blue!80!black] (\tickGr, 2.74+\barH/2) circle (1.5pt);
\fill[blue!80!black]
  (\tickC, 2.44) rectangle (\tickE, 2.44+\barH);
\fill[blue!80!black] (\tickC, 2.44+\barH/2) circle (1.5pt);
\fill[blue!80!black] (\tickE, 2.44+\barH/2) circle (1.5pt);
\node[font=\small, blue!80!black, anchor=west]
  at (\tickAr+0.2, 2.59+\barH/2)
  {$\mathrm{XRPH}_0(M,c)$};
\draw[gray!40, dashed, line width=0.6pt]
  (\tickA, 2.29) -- (\tickAr, 2.29);
\fill[black!80]
  (\tickB, 2.02) rectangle (\tickGr, 2.02+\barH);
\fill[black!80] (\tickB,  2.02+\barH/2) circle (1.5pt);
\fill[black!80] (\tickGr, 2.02+\barH/2) circle (1.5pt);
\node[font=\small, black!80, anchor=east]
  at (\tickB-0.12, 2.02+\barH/2) {$+$};
\fill[black!80]
  (\tickD, 1.72) rectangle (\tickFr, 1.72+\barH);
\fill[black!80] (\tickD,  1.72+\barH/2) circle (1.5pt);
\fill[black!80] (\tickFr, 1.72+\barH/2) circle (1.5pt);
\node[font=\small, black!80, anchor=east]
  at (\tickD-0.12, 1.72+\barH/2) {$-$};
\fill[black!80]
  (\tickC, 1.42) rectangle (\tickE, 1.42+\barH);
\fill[black!80] (\tickC, 1.42+\barH/2) circle (1.5pt);
\fill[black!80] (\tickE, 1.42+\barH/2) circle (1.5pt);
\node[font=\small, black!80, anchor=east]
  at (\tickC-0.12, 1.42+\barH/2) {$+$};
\node[font=\small, black!80, anchor=west]
  at (\tickAr+0.2, 1.72+\barH/2)
  {$\mathrm{XRPH}_0(\partial M,c)$};
\draw[gray!40, dashed, line width=0.6pt]
  (\tickA, 1.27) -- (\tickAr, 1.27);
\fill[green!50!black]
  (\tickA, 1.00) rectangle (\tickRr, 1.00+\barH);
\fill[green!50!black] (\tickA,  1.00+\barH/2) circle (1.5pt);
\fill[green!50!black] (\tickRr, 1.00+\barH/2) circle (1.5pt);
\fill[green!50!black]
  (\tickD, 0.70) rectangle (\tickFr, 0.70+\barH);
\fill[green!50!black] (\tickD,  0.70+\barH/2) circle (1.5pt);
\fill[green!50!black] (\tickFr, 0.70+\barH/2) circle (1.5pt);
\node[font=\small, green!50!black, anchor=west]
  at (\tickAr+0.2, 0.85+\barH/2)
  {$\mathrm{XRPH}_0(L,c)$};

\end{tikzpicture}
    \makebox[\textwidth][l]{\small\textbf{(c)} Extended radial persistence modules in homology degree $0$ for $M$, $\partial M$ and $L$.}

    \vspace{1em}

\begin{tikzpicture}[x=1cm, y=1cm]

\def\tickA{0.000}   
\def\tickB{1.016}   
\def\tickC{2.258}   
\def\tickD{2.710}   
\def\tickE{3.952}   
\def\tickF{5.194}   
\def\tickG{5.984}   
\def\tickR{7.000}   
\def\tickRr{7.600}  
\def\tickGr{8.616}  
\def\tickFr{9.406}  
\def\tickEr{10.648} 
\def\tickDr{11.890} 
\def\tickCr{12.342} 
\def\tickBr{13.584} 
\def\tickAr{14.600} 

\def\barH{0.12}
\pgfmathsetmacro{\axisline}{1.0 - 0.15}  
\draw[gray!60, line width=0.8pt]
  (\tickA,  \axisline) -- (\tickR,  \axisline);
\draw[gray!60, line width=0.8pt]
  (\tickRr, \axisline) -- (\tickAr, \axisline);
\draw[gray!50, dotted, line width=0.8pt]
  (7.3, 0.70) -- (7.3, 3.5);
\node[font=\small, text=gray] at (3.5,  3.35) {ordinary};
\node[font=\small, text=gray] at (11.1, 3.35) {relative};
\foreach \x in {
  \tickA, \tickB, \tickC, \tickD, \tickE, \tickF, \tickG, \tickR,
  \tickRr, \tickGr, \tickFr, \tickEr, \tickDr, \tickCr, \tickBr, \tickAr
}{
  \draw[gray!60] (\x, \axisline-0.07) -- (\x, \axisline+0.07);
}
\node[font=\tiny, rotate=270, anchor=west] at (\tickA,  \axisline-0.07) {$(0,\mathrm{Ord})$};
\node[font=\tiny, rotate=270, anchor=west] at (\tickB,  \axisline-0.07) {$(r_1,\mathrm{Ord})$};
\node[font=\tiny, rotate=270, anchor=west] at (\tickC,  \axisline-0.07) {$(r_2,\mathrm{Ord})$};
\node[font=\tiny, rotate=270, anchor=west] at (\tickD,  \axisline-0.07) {$(r_3,\mathrm{Ord})$};
\node[font=\tiny, rotate=270, anchor=west] at (\tickE,  \axisline-0.07) {$(r_4,\mathrm{Ord})$};
\node[font=\tiny, rotate=270, anchor=west] at (\tickF,  \axisline-0.07) {$(r_5,\mathrm{Ord})$};
\node[font=\tiny, rotate=270, anchor=west] at (\tickG,  \axisline-0.07) {$(r_6,\mathrm{Ord})$};
\node[font=\tiny, rotate=270, anchor=west] at (\tickR,  \axisline-0.07) {$(R,\mathrm{Ord})$};
\node[font=\tiny, rotate=270, anchor=west] at (\tickRr, \axisline-0.07) {$(R,\mathrm{Rel})$};
\node[font=\tiny, rotate=270, anchor=west] at (\tickGr, \axisline-0.07) {$(r_6,\mathrm{Rel})$};
\node[font=\tiny, rotate=270, anchor=west] at (\tickFr, \axisline-0.07) {$(r_5,\mathrm{Rel})$};
\node[font=\tiny, rotate=270, anchor=west] at (\tickEr, \axisline-0.07) {$(r_4,\mathrm{Rel})$};
\node[font=\tiny, rotate=270, anchor=west] at (\tickDr, \axisline-0.07) {$(r_3,\mathrm{Rel})$};
\node[font=\tiny, rotate=270, anchor=west] at (\tickCr, \axisline-0.07) {$(r_2,\mathrm{Rel})$};
\node[font=\tiny, rotate=270, anchor=west] at (\tickBr, \axisline-0.07) {$(r_1,\mathrm{Rel})$};
\node[font=\tiny, rotate=270, anchor=west] at (\tickAr, \axisline-0.07) {$(0,\mathrm{Rel})$};
\fill[blue!80!black]
  (\tickF, 2.74) rectangle (\tickDr, 2.74+\barH);
\fill[blue!80!black] (\tickF,  2.74+\barH/2) circle (1.5pt);
\fill[blue!80!black] (\tickDr, 2.74+\barH/2) circle (1.5pt);
\node[font=\small, blue!80!black, anchor=west]
  at (\tickAr+0.2, 2.74+\barH/2)
  {$\mathrm{XRPH}_1(M,c)$};
\draw[gray!40, dashed, line width=0.6pt]
  (\tickA, 2.59) -- (\tickAr, 2.59);
\fill[black!80]
  (\tickG, 2.32) rectangle (\tickBr, 2.32+\barH);
\fill[black!80] (\tickG,  2.32+\barH/2) circle (1.5pt);
\fill[black!80] (\tickBr, 2.32+\barH/2) circle (1.5pt);
\node[font=\small, black!80, anchor=east]
  at (\tickG-0.12, 2.32+\barH/2) {$-$};
\fill[black!80]
  (\tickF, 2.02) rectangle (\tickDr, 2.02+\barH);
\fill[black!80] (\tickF,  2.02+\barH/2) circle (1.5pt);
\fill[black!80] (\tickDr, 2.02+\barH/2) circle (1.5pt);
\node[font=\small, black!80, anchor=east]
  at (\tickF-0.12, 2.02+\barH/2) {$+$};
\fill[black!80]
  (\tickEr, 1.72) rectangle (\tickCr, 1.72+\barH);
\fill[black!80] (\tickEr, 1.72+\barH/2) circle (1.5pt);
\fill[black!80] (\tickCr, 1.72+\barH/2) circle (1.5pt);
\node[font=\small, black!80, anchor=east]
  at (\tickEr-0.12, 1.72+\barH/2) {$+$};
\node[font=\small, black!80, anchor=west]
  at (\tickAr+0.2, 2.02+\barH/2)
  {$\mathrm{XRPH}_1(\partial M,c)$};
\draw[gray!40, dashed, line width=0.6pt]
  (\tickA, 1.57) -- (\tickAr, 1.57);
\fill[green!50!black]
  (\tickG, 1.30) rectangle (\tickRr, 1.30+\barH);
\fill[green!50!black] (\tickG,  1.30+\barH/2) circle (1.5pt);
\fill[green!50!black] (\tickRr, 1.30+\barH/2) circle (1.5pt);
\fill[green!50!black]
  (\tickEr, 1.00) rectangle (\tickCr, 1.00+\barH);
\fill[green!50!black] (\tickEr, 1.00+\barH/2) circle (1.5pt);
\fill[green!50!black] (\tickCr, 1.00+\barH/2) circle (1.5pt);
\node[font=\small, green!50!black, anchor=west]
  at (\tickAr+0.2, 1.15+\barH/2)
  {$\mathrm{XRPH}_1(L,c)$};

\end{tikzpicture}    
    \makebox[\textwidth][l]{\small\textbf{(d)} Extended radial persistence modules in homology degree $1$ for $M$, $\partial M$ and $L$.}
\end{center}
\caption{
 \emph{Location of $c$:} In this example $c$ is in the infinite component of $\R^2\backslash M$. We can see this from the signed annotation of $\mathrm{XRPH}_0(\partial M,c)$. The earliest birth in $\partial M$ has a positive sign which  implies that $c\notin M$. This bar also has the earliest death of all the essential bars which means that $c$ is in the infinite component of $\R^2 \backslash M$. In this scenario it is easy to read off the extended persistent homology both in homology degree $0$ and in homology degree $1$. 
\emph{Obtaining $\mathrm{XRPH}_0(M,c)$ from the signed annotation of $\mathrm{XRPH}_0(\partial M,c)$:} 
The interval decomposition of $\mathrm{XRPH}_0(M,c)$ consists of the intervals in $\mathrm{XRPH}_0(\partial M,c)$ with a positive sign attached to the birth. 
\emph{Obtaining $\mathrm{XRPH}_1(M,c)$ from the signed annotation of $\mathrm{XRPH}_1(\partial M,c)$:} 
The interval decomposition of $\mathrm{XRPH}_1(M,c)$ consists of the intervals in $\mathrm{XRPH}_1(\partial M,c)$ with a positive sign attached to the birth. 
\emph{Obtaining $\mathrm{XRPH}_2(M,c)$:}  As $c\notin M$ we know that $\mathrm{XRPH}_2(M,c)=0$. 
}\label{fig:four-panel}
\end{figure}

\begin{figure}[htbp]
    \begin{center}
    \begin{tikzpicture}
        \node[anchor=north west] at (0, 0) {%
            \includegraphics[width=0.4 \textwidth]{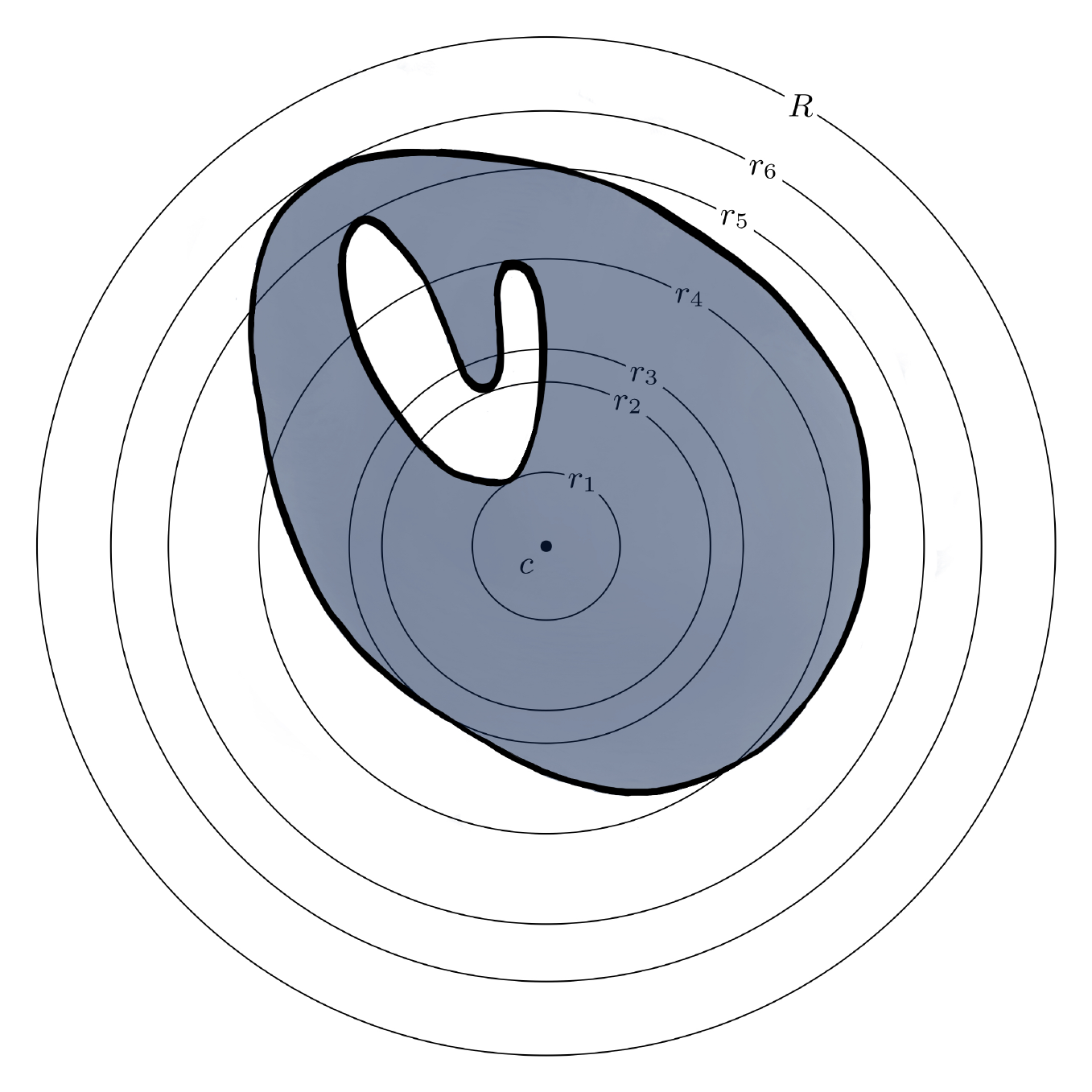}

        };
    \end{tikzpicture}%
    \hfill
    \begin{tikzpicture}
        \node[anchor=north west] at (0, 0) {%
            \includegraphics[width=0.4 \textwidth]{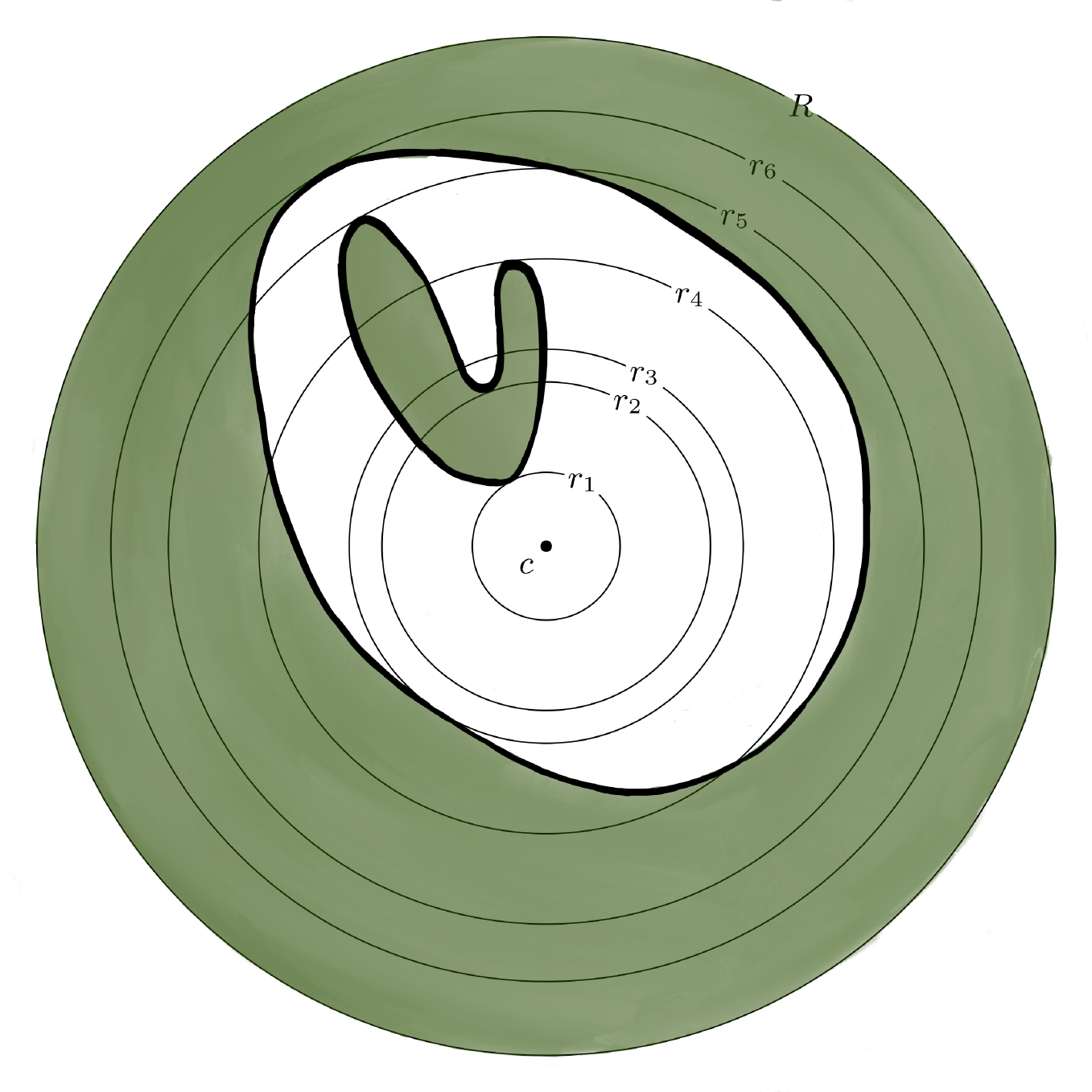}  
        };
    \end{tikzpicture}
    \makebox[0.48\textwidth][l]{\small\textbf{(a)} $M$}%
    \hfill
    \makebox[0.48\textwidth][l]{\small\textbf{(b)} $L=\overline{B(0,R)\backslash M}$}

    \vspace{1em}
\end{center}

\begin{tikzpicture}[x=1cm, y=1cm]
\def\tickA{0.000}   
\def\tickB{1.016}   
\def\tickC{2.258}   
\def\tickD{2.710}   
\def\tickE{3.952}   
\def\tickF{5.194}   
\def\tickG{5.984}   
\def\tickR{7.000}   
\def\tickRr{7.600}  
\def\tickGr{8.616}  
\def\tickFr{9.406}  
\def\tickEr{10.648} 
\def\tickDr{11.890} 
\def\tickCr{12.342} 
\def\tickBr{13.584} 
\def\tickAr{14.600} 

\def\barH{0.12}
\def\axisY{-0.3}
\draw[gray!60, line width=0.8pt]
  (\tickA, \axisY+0.7) -- (\tickR, \axisY+0.7);
\draw[gray!60, line width=0.8pt]
  (\tickRr, \axisY+0.7) -- (\tickAr, \axisY+0.7);
\draw[gray!50, dotted, line width=0.8pt]
  (7.3, \axisY-0.15) -- (7.3, 3.5);
\node[font=\small, text=gray] at (3.5,  3.35) {ordinary};
\node[font=\small, text=gray] at (11.1, 3.35) {relative};
\foreach \x in {
  \tickA, \tickB, \tickC, \tickD, \tickE, \tickF, \tickG, \tickR,
  \tickRr, \tickGr, \tickFr, \tickEr, \tickDr, \tickCr, \tickBr, \tickAr
}{
  \draw[gray!60] (\x, \axisY+0.63) -- (\x, \axisY+0.77);
}
\node[font=\tiny, rotate=270, anchor=west] at (\tickA,  \axisY+0.58) {$(0,\mathrm{Ord})$};
\node[font=\tiny, rotate=270, anchor=west] at (\tickB,  \axisY+0.58) {$(r_1,\mathrm{Ord})$};
\node[font=\tiny, rotate=270, anchor=west] at (\tickC,  \axisY+0.58) {$(r_2,\mathrm{Ord})$};
\node[font=\tiny, rotate=270, anchor=west] at (\tickD,  \axisY+0.58) {$(r_3,\mathrm{Ord})$};
\node[font=\tiny, rotate=270, anchor=west] at (\tickE,  \axisY+0.58) {$(r_4,\mathrm{Ord})$};
\node[font=\tiny, rotate=270, anchor=west] at (\tickF,  \axisY+0.58) {$(r_5,\mathrm{Ord})$};
\node[font=\tiny, rotate=270, anchor=west] at (\tickG,  \axisY+0.58) {$(r_6,\mathrm{Ord})$};
\node[font=\tiny, rotate=270, anchor=west] at (\tickR,  \axisY+0.58) {$(R,\mathrm{Ord})$};
\node[font=\tiny, rotate=270, anchor=west] at (\tickRr, \axisY+0.58) {$(R,\mathrm{Rel})$};
\node[font=\tiny, rotate=270, anchor=west] at (\tickGr, \axisY+0.58) {$(r_6,\mathrm{Rel})$};
\node[font=\tiny, rotate=270, anchor=west] at (\tickFr, \axisY+0.58) {$(r_5,\mathrm{Rel})$};
\node[font=\tiny, rotate=270, anchor=west] at (\tickEr, \axisY+0.58) {$(r_4,\mathrm{Rel})$};
\node[font=\tiny, rotate=270, anchor=west] at (\tickDr, \axisY+0.58) {$(r_3,\mathrm{Rel})$};
\node[font=\tiny, rotate=270, anchor=west] at (\tickCr, \axisY+0.58) {$(r_2,\mathrm{Rel})$};
\node[font=\tiny, rotate=270, anchor=west] at (\tickBr, \axisY+0.58) {$(r_1,\mathrm{Rel})$};
\node[font=\tiny, rotate=270, anchor=west] at (\tickAr, \axisY+0.58) {$(0,\mathrm{Rel})$};
\fill[blue!80!black]
  (\tickA, 2.6) rectangle (\tickGr, 2.6+\barH);
\fill[blue!80!black] (\tickA,  2.6+\barH/2) circle (1.5pt);
\fill[blue!80!black] (\tickGr, 2.6+\barH/2) circle (1.5pt);
\fill[blue!80!black]
  (\tickC, 2.3) rectangle (\tickE, 2.3+\barH);
\fill[blue!80!black] (\tickC, 2.3+\barH/2) circle (1.5pt);
\fill[blue!80!black] (\tickE, 2.3+\barH/2) circle (1.5pt);
\node[font=\small, blue!80!black, anchor=west]
  at (\tickAr+0.2, 2.45+\barH/2)
  {$\mathrm{XRPH}_0(M,c)$};
\draw[gray!40, dashed, line width=0.6pt]
  (\tickA, 2.15) -- (\tickAr, 2.15);
\fill[black!80]
  (\tickB, 1.9) rectangle (\tickFr, 1.9+\barH);
\fill[black!80] (\tickB,  1.9+\barH/2) circle (1.5pt);
\fill[black!80] (\tickFr, 1.9+\barH/2) circle (1.5pt);
\node[font=\small, black!80, anchor=east]
  at (\tickB-0.12, 1.9+\barH/2) {$-$};
\fill[black!80]
  (\tickD, 1.6) rectangle (\tickGr, 1.6+\barH);
\fill[black!80] (\tickD,  1.6+\barH/2) circle (1.5pt);
\fill[black!80] (\tickGr, 1.6+\barH/2) circle (1.5pt);
\node[font=\small, black!80, anchor=east]
  at (\tickD-0.12, 1.6+\barH/2) {$-$};
\fill[black!80]
  (\tickC, 1.3) rectangle (\tickE, 1.3+\barH);
\fill[black!80] (\tickC, 1.3+\barH/2) circle (1.5pt);
\fill[black!80] (\tickE, 1.3+\barH/2) circle (1.5pt);
\node[font=\small, black!80, anchor=east]
  at (\tickC-0.12, 1.3+\barH/2) {$+$};
\node[font=\small, black!80, anchor=west]
  at (\tickAr+0.2, 1.6+\barH/2)
  {$\mathrm{XRPH}_0(\partial M,c)$};
\draw[gray!40, dashed, line width=0.6pt]
  (\tickA, 1.05) -- (\tickAr, 1.05);
\fill[green!50!black]
  (\tickB, 0.8) rectangle (\tickFr, 0.8+\barH);
\fill[green!50!black] (\tickB,  0.8+\barH/2) circle (1.5pt);
\fill[green!50!black] (\tickFr, 0.8+\barH/2) circle (1.5pt);
\fill[green!50!black]
  (\tickD, 0.5) rectangle (\tickRr, 0.5+\barH);
\fill[green!50!black] (\tickD,  0.5+\barH/2) circle (1.5pt);
\fill[green!50!black] (\tickRr, 0.5+\barH/2) circle (1.5pt);
\node[font=\small, green!50!black, anchor=west]
  at (\tickAr+0.2, 0.65+\barH/2)
  {$\mathrm{XRPH}_0(L,c)$};

\end{tikzpicture}
    \makebox[\textwidth][l]{\small\textbf{(c)} Extended radial persistence modules in homology degree $0$ for $M$, $\partial M$ and $L$.}

    \vspace{1em}

\begin{tikzpicture}[x=1cm, y=1cm]

\def\tickA{0.000}   
\def\tickB{1.016}   
\def\tickC{2.258}   
\def\tickD{2.710}   
\def\tickE{3.952}   
\def\tickF{5.194}   
\def\tickG{5.984}   
\def\tickR{7.000}   
\def\tickRr{7.600}  
\def\tickGr{8.616}  
\def\tickFr{9.406}  
\def\tickEr{10.648} 
\def\tickDr{11.890} 
\def\tickCr{12.342} 
\def\tickBr{13.584} 
\def\tickAr{14.600} 

\def\barH{0.12}
\pgfmathsetmacro{\axisline}{1.08 - 0.15}  
\draw[gray!60, line width=0.8pt]
  (\tickA,  \axisline) -- (\tickR,  \axisline);
\draw[gray!60, line width=0.8pt]
  (\tickRr, \axisline) -- (\tickAr, \axisline);
\draw[gray!50, dotted, line width=0.8pt]
  (7.3, 0.75) -- (7.3, 3.5);
\node[font=\small, text=gray] at (3.5,  3.35) {ordinary};
\node[font=\small, text=gray] at (11.1, 3.35) {relative};
\foreach \x in {
  \tickA, \tickB, \tickC, \tickD, \tickE, \tickF, \tickG, \tickR,
  \tickRr, \tickGr, \tickFr, \tickEr, \tickDr, \tickCr, \tickBr, \tickAr
}{
  \draw[gray!60] (\x, \axisline-0.07) -- (\x, \axisline+0.07);
}
\node[font=\tiny, rotate=270, anchor=west] at (\tickA,  \axisline-0.07) {$(0,\mathrm{Ord})$};
\node[font=\tiny, rotate=270, anchor=west] at (\tickB,  \axisline-0.07) {$(r_1,\mathrm{Ord})$};
\node[font=\tiny, rotate=270, anchor=west] at (\tickC,  \axisline-0.07) {$(r_2,\mathrm{Ord})$};
\node[font=\tiny, rotate=270, anchor=west] at (\tickD,  \axisline-0.07) {$(r_3,\mathrm{Ord})$};
\node[font=\tiny, rotate=270, anchor=west] at (\tickE,  \axisline-0.07) {$(r_4,\mathrm{Ord})$};
\node[font=\tiny, rotate=270, anchor=west] at (\tickF,  \axisline-0.07) {$(r_5,\mathrm{Ord})$};
\node[font=\tiny, rotate=270, anchor=west] at (\tickG,  \axisline-0.07) {$(r_6,\mathrm{Ord})$};
\node[font=\tiny, rotate=270, anchor=west] at (\tickR,  \axisline-0.07) {$(R,\mathrm{Ord})$};
\node[font=\tiny, rotate=270, anchor=west] at (\tickRr, \axisline-0.07) {$(R,\mathrm{Rel})$};
\node[font=\tiny, rotate=270, anchor=west] at (\tickGr, \axisline-0.07) {$(r_6,\mathrm{Rel})$};
\node[font=\tiny, rotate=270, anchor=west] at (\tickFr, \axisline-0.07) {$(r_5,\mathrm{Rel})$};
\node[font=\tiny, rotate=270, anchor=west] at (\tickEr, \axisline-0.07) {$(r_4,\mathrm{Rel})$};
\node[font=\tiny, rotate=270, anchor=west] at (\tickDr, \axisline-0.07) {$(r_3,\mathrm{Rel})$};
\node[font=\tiny, rotate=270, anchor=west] at (\tickCr, \axisline-0.07) {$(r_2,\mathrm{Rel})$};
\node[font=\tiny, rotate=270, anchor=west] at (\tickBr, \axisline-0.07) {$(r_1,\mathrm{Rel})$};
\node[font=\tiny, rotate=270, anchor=west] at (\tickAr, \axisline-0.07) {$(0,\mathrm{Rel})$};
\fill[blue!80!black]
  (\tickF, 2.9) rectangle (\tickDr, 2.9+\barH);
\fill[blue!80!black] (\tickF,  2.9+\barH/2) circle (1.5pt);
\fill[blue!80!black] (\tickDr, 2.9+\barH/2) circle (1.5pt);
\node[font=\small, blue!80!black, anchor=west]
  at (\tickAr+0.2, 2.9+\barH/2)
  {$\mathrm{XRPH}_1(M,c)$};
\draw[gray!40, dashed, line width=0.6pt]
  (\tickA, 2.65) -- (\tickAr, 2.65);
\fill[black!80]
  (\tickF, 2.4) rectangle (\tickBr, 2.4+\barH);
\fill[black!80] (\tickF,  2.4+\barH/2) circle (1.5pt);
\fill[black!80] (\tickBr, 2.4+\barH/2) circle (1.5pt);
\node[font=\small, black!80, anchor=east]
  at (\tickF-0.12, 2.4+\barH/2) {$+$};
\fill[black!80]
  (\tickG, 2.1) rectangle (\tickDr, 2.1+\barH);
\fill[black!80] (\tickG,  2.1+\barH/2) circle (1.5pt);
\fill[black!80] (\tickDr, 2.1+\barH/2) circle (1.5pt);
\node[font=\small, black!80, anchor=east]
  at (\tickG-0.12, 2.1+\barH/2) {$-$};
\fill[black!80]
  (\tickEr, 1.8) rectangle (\tickCr, 1.8+\barH);
\fill[black!80] (\tickEr, 1.8+\barH/2) circle (1.5pt);
\fill[black!80] (\tickCr, 1.8+\barH/2) circle (1.5pt);
\node[font=\small, black!80, anchor=east]
  at (\tickEr-0.12, 1.8+\barH/2) {$+$};
\node[font=\small, black!80, anchor=west]
  at (\tickAr+0.2, 2.1+\barH/2)
  {$\mathrm{XRPH}_1(\partial M,c)$};
\draw[gray!40, dashed, line width=0.6pt]
  (\tickA, 1.55) -- (\tickAr, 1.55);
\fill[green!50!black]
  (\tickG, 1.3) rectangle (\tickRr, 1.3+\barH);
\fill[green!50!black] (\tickG,  1.3+\barH/2) circle (1.5pt);
\fill[green!50!black] (\tickRr, 1.3+\barH/2) circle (1.5pt);
\fill[green!50!black]
  (\tickEr, 1.08) rectangle (\tickCr, 1.08+\barH);
\fill[green!50!black] (\tickEr, 1.08+\barH/2) circle (1.5pt);
\fill[green!50!black] (\tickCr, 1.08+\barH/2) circle (1.5pt);
\node[font=\small, green!50!black, anchor=west]
  at (\tickAr+0.2, 1.19+\barH/2)
  {$\mathrm{XRPH}_1(L,c)$};

\end{tikzpicture}
    \makebox[\textwidth][l]{\small\textbf{(d)} Extended radial persistence modules in homology degree $1$ for $M$, $\partial M$ and $L$.}
 \caption{
 \emph{Location of $c$:} In this example $c$ is inside $M$. We can read this from the sign annotation of  $\mathrm{XRPH}_0(\partial M,c)$: the earliest birth for $\partial M$ has a negative sign. 
 This bar does not have the earliest death of all essential bars which means the point on $\partial M$ closest to $c$ is not on the exterior boundary of $M$.
\emph{Obtaining $\mathrm{XRPH}_0(M)$ from the signed annotation of $\mathrm{XRPH}_0(\partial M)$:} 
As $c\in M$ and $r_6$ is the maximum value attained in $M$,  we know that the only essential class in $\mathrm{XRPH}_0(M)$ is  the interval $[(0, \mathrm{Ord})), (r_6, \mathrm{Rel}))$.
$\mathrm{Ord}_0(M, \rho_c)$ has the intervals in $\mathrm{Ord}_0(\partial M,\rho_c)$ with a positive sign annotation on the birth.
\emph{Obtaining $\mathrm{XRPH}_1(M)$ from the signed annotation of $\mathrm{XRPH}_1(\partial M)$:} 
In this scenario we must trace a ``cascade'' of death values to find the degree-1 extended radial persistent homology from the annotated extended radial persistent homology of $\partial M$. The death cascade sequence is given in Definition \ref{def:cascadeSeq}. Geometrically, we can see that the pair $(M,M^{r_3})$ has the combinatorial structure of a punctured 2-sphere, implying that the component in $\mathrm{XRPH}_1(M,c)$ born at $(r_5,\mathrm{Ord})$ dies at $(r_3,\mathrm{Rel})$.
\emph{Obtaining $\mathrm{XRPH}_2(M)$:} As $c\in M$ and $r_1$ is the minimum value attained in $\partial M$, we know that $\mathrm{XRPH}_2(M)$ is an interval module over $[(r_1, \mathrm{Rel})), (0, \mathrm{Rel}))$.} 
\label{fig:four-panel2}
\end{figure}
\medskip 

We now set up notation used in the rest of this section. We work with radial distance functions $\rho_c:\R^n\to\R$ with a given centre $c\in \R^n$. For a subset $A\subset \R^n$ we write $\rho_c^A\coloneqq \rho_c\vert_A$ for the restriction of $\rho_c$ to $A$. Furthermore, for $\alpha\in \R$, we write $A_\alpha \coloneqq \left(\rho_c^A\right)^{-1}(-\infty,\alpha]$ for the sublevel set and $A^\alpha \coloneqq \left(\rho_c^A\right)^{-1}[\alpha,\infty)$ for the superlevel set.
We use the notation $S_k^A$ for the set of intervals in the interval decomposition of $\mathrm{XRPH}_k(A,c)$, so that 
$\mathrm{XRPH}_k(A,c)=\bigoplus_{[b,d)\in S_k^{A}}\mathcal{I}_{[b,d)}$. 
We write $\text{Crit}(\rho_c^M, k)$ to denote the set of critical points of $\rho_c^{\text{int}(M)}$ with index $k$. Furthermore, denote by $\text{Crit}(\rho_c^M,(k,\eta))$ the set of critical points of $\rho_c^{\partial M}$ with index $(k,\eta)$. For a non-degenerate critical point $p\in \partial M$, we call $\eta$ its \emph{sign}, and denote this by $\eta = \text{sgn}(\rho_c^M, p)$. 
Finally, suppose $$\text{XRPH}(M,c) = \text{Ord}_*(M,\rho_c)\oplus \text{Rel}_*(M,\rho_c) \oplus \text{Ess}_*(M,\rho_c)$$ is the radial extended persistence module of $M$; we define the following notation for the set of birth  parameters of intervals in $\text{XRPH}(M,c)$:
$$\mathfrak{b}_k^\text{ord}(M,\rho_c) \coloneqq \mathfrak{b}\left(\text{Ord}_k(M,\rho_c)\oplus \text{Ess}_k(M,\rho_c)\right),$$ and
$$\mathfrak{b}_k^\text{rel}(M,\rho_c)\coloneqq \mathfrak{b}\left(\text{Rel}_k(M,\rho_c)\right).$$ 
Similarly, we define the set of death parameters of intervals as 
$$\mathfrak{d}_k^\text{ord}(M,\rho_c)\coloneqq \mathfrak{d}\left(\text{Ord}_k(M,\rho_c)\right),$$ and 
$$\mathfrak{d}_k^\text{rel}(M,\rho_c) \coloneqq \mathfrak{d}\left(\text{Rel}_k(M,\rho_c)\oplus \text{Ess}_k(M,\rho_c)\right).$$ Furthermore, let $\mathfrak{b}_k(M,\rho_c) = \mathfrak{b}_k^\mathrm{ord}(M,\rho_c)\cup \mathfrak{b}_k^{\mathrm{rel}}(M,\rho_c)$ and $\mathfrak{d}_k(M,\rho_c) = \mathfrak{d}_k^\mathrm{ord}(M,\rho_c)\cup \mathfrak{d}_k^{\mathrm{rel}}(M,\rho_c)$. This notation is consistent with that in \cite{KT2022}.

Theorem \ref{breakpoint} is key to proving our main results. Its proof can be found in \cite[Corollary 4.14]{KT2022}.  

\begin{theorem}\label{breakpoint}
Let $(M,\partial M)$ be an $n$-manifold with boundary. If $f:M\to \R$ is a Morse function then for all $k\geq 0$, we have
$$\mathfrak{b}_k^\text{ord}(M,f) \cup \mathfrak{d}_{k-1}^\text{ord}(M,f) = \{(f(p),\text{Ord}) \;\vert\; p\in \text{Crit}(f,k)\cup \text{Crit}(f,(k,+1))\},$$ and
\begin{align*}
    \mathfrak{b}_k^\text{rel}(M,f) \cup \mathfrak{d}_{k-1}^\text{rel}(M,f) = \{(f(p),\text{Rel}) \;\vert\; p&\in \text{Crit}(f,n-k)\cup \text{Crit}(f,(n-k-1,-1))\}.
\end{align*}
\end{theorem}
Under reasonable genericity conditions on the radial function, every $\rho_c$ is  arbitrarily close to a Morse function with the equivalent extended persistent homology, i.e., the corresponding XRPHs have the same characterisation. By ``arbitrarily close'' we mean that there is an alternative choice $c'$ close to $c$ with $\rho_{c'}$ Morse.
The important quality we use is finitely many critical values and for them to be distinct. To remove this assumption about distinct critical values we could invoke stability. It can be easily shown that the interleaving distance between the modules for different centers $\rho_c$ and $\rho_{c'}$ is bounded by the distance $\|c-c'\|_2$. It will be convenient to keep the Morse assumption with its distinct critical values for most of the paper and only remove this assumption via a stability argument once the theory is fully established in the generic case.

A very important tool throughout is the Meyer-Vietoris long exact sequence which allows us to relate the homology of sublevel (superlevel) sets of $\partial M$ with that of sublevel (superlevel) sets of $M$. To that end, choose $R>0$ such that $M\subset B(c,R-2\delta)$ (for $\delta>0$ small), where $B(c,R-2\delta)$ is the open ball of radius $R-2\delta$ centered at $c\in\R^n$. 
Let $L \coloneqq \overline{B(\hat{c},R)}\backslash \text{int}\,M$, for some $\hat{c}$ with $0<\|c-\hat{c}\|_2<\delta$. Note that $\partial M = M\cap L$. Let $R^+=R+\|c-\hat{c}\|_2$ and $R^-=R-\|c-\hat{c}\|_2$; these are the maximum and minimum values of $\rho_c$ over $\partial B(\hat{c}, R)$. 

Since we assume $c$ is generic meaning $\rho_c$ is a Morse function on $M$, there exists some $\epsilon>0$ such that all critical values of $\rho_c^M$ are at least $\epsilon$ apart and that the largest critical value is strictly less than $R-\epsilon$. It follows that $ 0 \leq \inf(\rho_c(M)) < \sup(\rho_c(M)) < R-\epsilon$ since $\rho_c$ takes non-negative values.

For $s>0$, we consider the sublevel sets of $\rho_c$ restricted to the three subsets $M,\partial M$ and $L$ of $\R^n$, namely $M_s$, $(\partial M)_s$ and $L_s$ respectively. Then we have $(\partial M)_s=M_s\cap L_s$ and $M_s\cup L_s = \rho_c^{-1}(-\infty, s]\cap \overline{B(c,R)}=\overline{B(c,s)}$ and the Mayer-Vietoris long exact sequence (LES) is
\begin{align}\label{LESord}
\begin{split}
    \cdots\overset{}{\longrightarrow}H_{k+1}(M_s\cup L_s)\longrightarrow H_k((\partial M)_s)\overset{}{\longrightarrow} H_k(M_s)\oplus H_k(L_s)\overset{}{\longrightarrow} \\
   H_k(M_s\cup L_s)\overset{}{\longrightarrow}
    \cdots\longrightarrow H_0(M_s\cup L_s)\longrightarrow 0.
\end{split}
\end{align}
There is also a relative version of the Mayer-Vietoris long exact sequence: 
\begin{align}\label{LESrel}
\begin{split}
    \cdots\overset{}{\longrightarrow}H_{k+1}(M\cup L, M^s\cup L^s)\longrightarrow H_k((\partial M), (\partial M)^s)\overset{}{\longrightarrow} H_k(M,M^s)\oplus H_k(L,L^s)\\ \overset{}{\longrightarrow}
   H_k(M\cup L, M^s\cup L^s)\overset{}{\longrightarrow}
    \cdots\longrightarrow H_0(M\cup L, M^s\cup L^s)\longrightarrow 0.
\end{split} \end{align} 

Observe that $H_k(M_s\cup L_s) = 0$ for all $k>0$. LES (\ref{LESord}) implies $H_k((\partial M)_s)\cong H_k(M_s)\oplus H_k(L_s)$. Similar relations can be shown from the relative LES (\ref{LESrel}). Together, they give a neat matching (Theorem \ref{thm:k}) between $\mathrm{XRPH}_k(M,c)$ and $\mathrm{XRPH}_k(\partial M,c)$ for the case $0<k<n-1$.

Before proving Theorem~\ref{thm:k} we state a simple relationship between critical points in $\partial M$, $M$ and $L$. 
\begin{lemma}\label{lem:swaps}
    If $p\in \partial M$ is a non-degenerate critical point in $\mathrm{Crit}(\rho_c^M,(k,\eta))$, then $p$ is a critical point in $\mathrm{Crit}(\rho_c^L,(k,-\eta)).$
\end{lemma}

\begin{theorem}\label{thm:k}
Let $0<k< n-1$. 
If $$\text{XRPH}_k(\partial M,\rho_c) = \bigoplus_{[b_i,d_i)\in S_k^{\partial M}} \mathcal{I}_{[b_i,d_i)}.$$ Let $J_k^M\subseteq S_k^{\partial M}$ be the subset of intervals $[b_i,d_i)$ such that  either $b_i = (\rho_c(p),\text{Ord})$ with $p\in\text{Crit}(\rho_c^M,(k,+1))$, or $b_i = (\rho_c(p),\text{Rel})$ with $p\in\text{Crit}(\rho_c^M,(n-k-1,-1))\}.$ 
Then
$$\text{XRPH}_k(M,c) = \bigoplus_{[b_i,d_i)\in J_k^M} \mathcal{I}_{[b_i,d_i)}.$$ 
\end{theorem}
\begin{proof}
For $0<k<n-1$, since $\partial M\subseteq M$ and $\partial M\subseteq L$, there is an induced morphism on persistence modules:

$$\phi_k: \text{XRPH}_k(\partial M,\rho_c)\to\text{XRPH}_k(M,c)\oplus\text{XRPH}_k(L,\rho_c).$$ 

Let $\phi_{(t,\mathrm{Ord})}$ and $\phi_{(t,\mathrm{Rel})}$ be the restrictions of $\phi_k$ to the vector space at parameter $(t, \mathrm{Ord})$ and $(t,\mathrm{Rel})$ respectively. Then LES (\ref{LESord}) and (\ref{LESrel}) above show that $\phi_{(t,\text{Ord})}$ and $\phi_{(t,\text{Rel})}$ are both isomorphisms for all $t$. Since the restriction to each parameter in $\Theta$ is an isomorphism, we conclude that $\phi_k$ is an isomorphism of persistence modules.

As $\phi_k$ is an isomorphism we can infer that the intervals of $\text{XRPH}_k(M,c)$ are a subset of the intervals of  $\text{XRPH}_k(\partial M,\rho_c)$. The identification can be done using Theorem \ref{breakpoint} together with Lemma \ref{lem:swaps}.

Explicitly, if $b_i\in \mathfrak{b}_k(M,c)$ is a critical value achieved at $p\in \partial M$, then Theorem \ref{breakpoint} tells us that $p\in \mathrm{Crit}(\rho_c^M,(k,+1))\cup \mathrm{Crit}(\rho_c^M,(n-k-1,-1))$. On the other hand, if $b_i\in \mathfrak{b}_k(L,\rho_c)$ is a critical value achieved at $q\in \partial M$, then applying Theorem \ref{breakpoint} to $L$, it must be that $q\in \mathrm{Crit}(\rho_c^L,(k,+1))\cup \mathrm{Crit}(\rho_c^L,(n-k-1,-1))$. Now Lemma \ref{lem:swaps} tells us that $q$ is a critical point of $\rho_c^M$ with opposite sign. So we have $q\in \mathrm{Crit}(\rho_c^M,(k,-1))\cup \mathrm{Crit}(\rho_c^L,(n-k-1,+1))$ instead.
\end{proof}

Theorem \ref{thm:k} covers all the matchings except for the cases of $k=0, n-1$ and $n$, which we will now consider. We start by establishing a matching (Lemma \ref{lem:births0}) between birth and death parameters for $M$ and $L$ and those for $\partial M$. Before that, we introduce some notations.

Let $r_{\min} \coloneqq \min_{p\in \partial M}\rho_c(p)$ and $r_{\max}\coloneqq \max_{p\in \partial M}\rho_c(p)$ be the minimum and maximum values of $\rho_c$ achieved on the boundary $\partial M$. We will further differentiate the different components of $\partial M$.

\begin{defn}[Interior and exterior boundary component of $(M,\partial M)$]\label{bndComp}
Suppose $M$ is connected. Let $A$ be a connected component of $\partial M$. Then $A$ is an \emph{interior boundary component} if $M\backslash A$ is contained in the unbounded connected component of $\R^n\backslash A$. We call $A$ an \emph{exterior boundary component} if $M\backslash A$ is contained in the bounded component of $\R^n\backslash A$. 
\end{defn}

Since $M$ is connected, it has exactly one exterior boundary component. We further define $r_{\min}^\mathrm{ext}\coloneqq \min_{p\in A^{\mathrm{ext}}}\rho_c(p),$ and $r_{\max}\coloneqq \max_{p\in A^\mathrm{ext}}\rho_c(p)$ to be the minimum and maximum values achieved on the exterior boundary $A^{\mathrm{ext}}$ of $M$.
Moreover, observe that $r_{\max}$ can only be achieved on $A^\mathrm{ext}$ and so $r_{\max}^{\mathrm{ext}} = r_{\max}$. 

\begin{lemma}\label{lem:births0}

We have equality of the following disjoint unions:
\begin{align*}
\mathfrak{b}_0(M,\rho_c)\sqcup  \mathfrak{b}_0(L,\rho_c)&= \mathfrak{b}_0(\partial M,\rho_c)\sqcup \{(0,\text{Ord})\};\\
\mathfrak{d}_0(M,\rho_c)\sqcup  \mathfrak{d}_0(L,\rho_c) &= \mathfrak{d}_0(\partial M,\rho_c)\sqcup \{(R^+,\text{Rel})\};\\
\mathfrak{b}_{n-1}(M,c)\sqcup  \mathfrak{b}_{n-1}(L,\rho_c)&= \mathfrak{b}_{n-1}(\partial M,\rho_c);\\
\mathfrak{d}_{n-1}(M,\rho_c)\sqcup  \mathfrak{d}_{n-1}(L,\rho_c)&= \big(\mathfrak{d}_{n-1}(\partial M,\rho_c)\sqcup \{(R^-,\text{Rel})\} \big)\backslash \{(r_{\min},\text{Rel})\};\\
\mathfrak{b}_n(M,c)\sqcup  \mathfrak{b}_n(L,\rho_c)&=\{(r_{\min},\text{Rel})\} ;\\
\mathfrak{d}_n(M,\rho_c)\sqcup  \mathfrak{d}_n(L,\rho_c) &= \{(0,\text{Rel})\};
\end{align*}

\end{lemma}
\begin{proof}

Consider the following part of LES (\ref{LESord}):
$$0\to H_0((\partial M)_s)\to H_0(M_s)\oplus H_0(L_s) \to H_0(M_s\cup L_s)\to 0.$$ As $H_0(M_s\cup L_s) = \Z_2$ for $s\geq 0$ we have that 
\begin{align}\label{rel2}\beta_0(M_s)+\beta_0(L_s) = \beta_0(M_s\cup L_s) +\beta_0((\partial M)_s) = \beta_0((\partial M)_s) + 1.\end{align}

Similarly, consider the following part of LES (\ref{LESrel}):
$$0\to H_0((\partial M),(\partial M)^s)\to H_0(M,M^s)\oplus H_0(L,L^s) \to H_0(M\cup L,M^s\cup L^s)\to 0.$$
Since $H_0(M\cup L,M^{R^+ +\epsilon}\cup L^{R^+ +\epsilon}) = H_0(M\cup L) = \Z_2$ and $H_0(M\cup L,M^{R^+-\epsilon}\cup L^{R^+-\epsilon})  = 0,$ and $(\partial M)^s = M^s=\emptyset$ for $R^+ -\epsilon\leq s\leq R^+ +\epsilon,$ it follows that
$$\beta_0(\partial M)=\beta_0(M)+\beta_0(L,L^{R^+ +\epsilon})-1,$$ and 
$$\beta_0(\partial M)=\beta_0(M)+\beta_0(L,L^{R^+ -\epsilon}).$$ Taking the difference of the two equations, we get
$$\beta_0(L,L^{R^+ -\epsilon})-\beta_0(L,L^{R^+ +\epsilon})=-1.$$ 

More generally $\beta_0(\partial M, (\partial M)^s)=\beta_0(M, M^s)+\beta_0(L,L^{s})$ for $0<s<R^+$. 

As the critical values are all at distinct parameters we can use this relation between the Betti numbers to relate the birth parameters and death parameters.
We have $$\mathfrak{b}_0(\partial M,\rho_c)\sqcup \{(0,\text{Ord})\} = \mathfrak{b}_0(M,c)\sqcup \mathfrak{b}_0(L,\rho_c)$$ 
and 
$$\mathfrak{d}_0(\partial M,\rho_c)\sqcup \{(R^+,\text{Rel})\} = \mathfrak{d}_0(M,c)\sqcup \mathfrak{d}_0(L,\rho_c).$$ 

Similarly, by considering the following part of LES (\ref{LESord}) for $s\ge 0$:
$$0\to H_{n-1}((\partial M)_s)\to H_{n-1}(M_s)\oplus H_{n-1}(L_s) \to 0,$$ where $H_{n-1}(M_s\cup L_s) = 0,$ we can show that 
$$\mathfrak{b}_{n-1}(M,c)\sqcup  \mathfrak{b}_{n-1}(L,\rho_c)= \mathfrak{b}_{n-1}(\partial M,\rho_c).$$

Now since $\partial M$ is an $(n-1)$-manifold, we have 
$$H_n(\partial M,(\partial M)^s) = 0,$$ for all $s\ge 0$. So the relevant part of long exact sequence (\ref{LESrel}) is
\begin{align}\label{LESYAY3}
\begin{split}
   0 \overset{}{\longrightarrow} H_n(M,M^s)\oplus H_n(L,L^s)\hookrightarrow H_n(M\cup L,M^s\cup L^s) \hspace{3cm}\\
  \hspace{2cm} \longrightarrow H_{n-1}((\partial M), (\partial M)^s)
  \twoheadrightarrow H_{n-1}(M,M^s)\oplus H_{n-1}(L,L^s)\overset{}{\longrightarrow} 0.
\end{split}
\end{align} 

Let $X$ and $Y$ be either $M$ or $L$ so $c\in X$ and $c\notin Y$. Then we have 
\begin{align}\label{step1}H_n(X,X^s)=\begin{cases} 0,&  r_{min} < s \\ \Z_2, & 0<s\leq r_{min}\\ 0, & s = 0\end{cases},\end{align} and 
\begin{align}\label{step2}H_n(Y,Y^s)=0, \quad \text{for all $s\ge 0$}. \end{align} 
This gives $$\mathfrak{b}_n(M,\rho_c)\sqcup  \mathfrak{b}_n(L,\rho_c)=\{(r_{\min},\text{Rel})\} ,\;
\mathfrak{d}_n(M,\rho_c)\sqcup  \mathfrak{d}_n(L,\rho_c) = \{(0,\text{Rel})\}.$$

Now when $0\leq s\leq r_{\min}$, note that $(\partial M)^s = \partial M$ by the definition of $r_{\min}$. Then the long exact sequence (\ref{LESYAY3}) tells us that 
\begin{equation}\label{useful1}
\beta_{n-1}(\partial M,(\partial M)^s) = \beta_{n-1}(M,M^s)+\beta_{n-1}(L,L^s)=0.\end{equation} When $r_{\min}< s< R^-$, on the other hand, we have $H_{n}(M,M^s)\oplus H_{n}(L,L^s)=0,$ and hence
\begin{equation}\label{useful2}\beta_{n-1}(\partial M,(\partial M)^s) = \beta_{n-1}(M,M^s)+\beta_{n-1}(L,L^s)+1.
\end{equation}
Combining (\ref{useful1}) and (\ref{useful2}), we have that $(r_{\min},\mathrm{Rel})\in \mathfrak{d}_{n-1}(\partial M,\rho_c).$ Since $\rho_c^M$ is Morse, we have that $(r_{\min},\mathrm{Rel})\not\in \mathfrak{d}_{n-1}(M,\rho_c)\sqcup \mathfrak{d}_{n-1}(L\rho_c)$. 
Finally, note that $(R^-,\mathrm{Rel})\in \mathfrak{d}_{n-1}(L,\rho_c)$ and $(R^-,\mathrm{Rel})\notin \mathfrak{d}_{n-1}(\partial M,\rho_c)$, a similar argument as above shows that we have that $$\mathfrak{d}_{n-1}(M,\rho_c)\sqcup  \mathfrak{d}_{n-1}(L,\rho_c)= \big(\mathfrak{d}_{n-1}(\partial M,\rho_c)\sqcup \{(R^-,\text{Rel})\} \big)\backslash \{(r_{\min},\text{Rel})\},$$ as desired.
\end{proof}

We now move on to consider the cases of $k=0$ and $k=n$. Here, as mentioned before, it matters if $c\in \mathrm{int} M$ or $c\not\in M$. So we consider them separately. 

Before giving the matching, we note first that the location of $c$ in relation to $M$ can be read off directly from the sign of the critical point with the earliest birth time in $\mathrm{XRPH}_0(\partial M,c).$

\begin{lemma}
Let $p$ be the critical point of $\rho_c^{\partial M}$ with the earliest birth time in $\mathrm{XRPH}_0(\partial M,c)$. Then $\mathrm{sgn}(\rho_c,p)=+1$ implies $c\notin M$, and $\mathrm{sgn}(\rho_c,p)=-1$ implies $c\in M$.
\end{lemma}
\begin{proof}
If $c\in M$, then the gradient must point out of $M$ at the earliest birth radius. This implies a negative sign. 
If $c\notin M$, then the gradient must point into $M$ at the earliest birth radius. This implies a positive sign.   
\end{proof}

As there is at most one birth in degree $n$ we can quickly read off $\mathrm{XRPH}_n(M,c)$.

\begin{lemma}
    If $c\in \text{int}\,M$, then $$\mathrm{XRPH}_n(M,c)=\mathcal{I}_{[(r_{\min},\text{Rel}),(0,\text{Rel}))}.$$ 
    If $c\notin  M$ then $$\mathrm{XRPH}_n(M,c)=0.$$
\end{lemma}
\begin{proof}

The only way to have a degree $n$-homology class for subsets of $\R^n$ under a radial function is with relative homology in the form $(B(c,r), \partial B(c,r))$. This means we only need to keep track of $r$ such that $B(c,r)\subset M$.  

For $c\in M$ we have $B(c,r)\subset M$ for all $r\leq r_{min}$ and hence the extended persistent homology is the interval module $\mathcal{I}_{[(r_{\min},\text{Rel}),(0,\text{Rel}))}$.

For $c\notin M$ we have $B(c,r)$ is not a subset of $M$ for all $r$. This implies that the extended persistent homology module is zero.

\end{proof}

At this point it will be useful to have a couple of lemmas for proof techniques that appear multiple times. A common method is to use the existence of an injective (or surjective) morphism $\phi$ between extended persistence modules to construct an induced matching $\hat{\phi}$ between the sets of births (or deaths). If we can find a subset of births (or deaths) $X$ such that $\hat{\phi}(X)=X$ then it is useful to know $\hat{\phi}|_X$ must be the identity.

\begin{lemma}\label{lem:identity}
Let $Y$ be a totally ordered finite set and $\hat{\phi}:Y\to Y$ a bijection such that $\hat{\phi}(y)\leq y$ for all $y\in Y$. Then $\hat{\phi}$ must be the identity. 
Symmetrically, if $\hat{\phi}(y)\geq y$ for all $y\in Y$. Then $\hat{\phi}$ must be the identity. 
\end{lemma}
\begin{proof}
This can be proved inductively from size of the set. Let $y_0$ be the smallest element. As $\hat{\phi}(y_0)\leq y_0$ and the only $y\in Y$ with $y\leq y_0$ is $y_0$ itself we conclude $\hat{\phi}(y_0)=y_0$. We then can consider $\hat{\phi}|_{Y\backslash \{y_0\}}:Y\backslash \{y_0\} \to Y\backslash \{y_0\}$ which satisfies the considition of the lemma and is a strictly smaller set. The inductive hypothesis thus says   $\hat{\phi}|_{Y\backslash \{y_0\}}$ is the identity map. This combined with $\hat{\phi}(y_0)=y_0$ implies $\hat{\phi}$ is the identity map.
\end{proof}

\begin{theorem}\label{thm:0cinM}
Assume that $c\in \text{int}\,M$. If $$\text{XRPH}_0(\partial M, c) = \bigoplus_{[b_i,d_i)\in S_0^{\partial M}} \mathcal{I}_{[b_i,d_i)}$$ and $J_0^M\subseteq S_0^{\partial M}$ be the subset of intervals $[b_i,d_i)$ such that $b_i = (\rho_c(p),\text{Ord})$ with $p\in\text{Crit}(\rho_c^M,(0,+1))$,
then $$\text{XRPH}_0(M,c) = \mathcal{I}_{[(0,\text{Ord}),(r_{\max},\text{Rel}))}\oplus \bigoplus_{[b_i,d_i)\in J^M_0} \mathcal{I}_{[b_i,d_i)}.$$
\end{theorem}

\begin{proof}
From Lemma \ref{lem:births0}, we have 
$\mathfrak{b}_0(\partial M,\rho_c)\sqcup \{(0,\text{Ord})\} = \mathfrak{b}_0(M,\rho_c)\sqcup \mathfrak{b}_0(L,\rho_c),$ which we will denote by $\mathfrak{b}$. Note that all degree-$0$ births appear in the ordinary parameter range.
By checking the ordinary Mayer Vietoris sequences (\ref{LESord}) and (\ref{LESrel}), we know that both maps
$$H_0((\partial M)_s)\to H_0(M_s)\oplus H_0(L_s),$$ and 
$$H_0(\partial M,(\partial M)^s)\to H_0(M,M^s)\oplus H_0(L,L^s)$$ are induced by inclusions and are injective for all $s\ge 0$. Hence, the morphism 
$$\phi_0:  \text{XRPH}_0(\partial M,c)\to\text{XRPH}_0(M,c)\oplus\text{XRPH}_0(L,\rho_c)$$ is injective. The map $\phi_0$ induces an injective map $\hat{\phi}_0: \mathfrak{b}_0(\partial M,\rho_c)\to\birth_0(M, \rho_c)\cup \birth_0(L, \rho_c)$ that matches the interval $[b,d)$ in the interval decomposition of $\text{XRPH}_0(\partial M, c)$ to the interval $[b',d)$ in the interval decomposition of $\text{XRPH}_0(M,c)\oplus \text{XRPH}_0(L,c)$ such that $b'\leq b$.

As $c\notin L$, the connected component of $L$ in the infinite exterior component of $\R^n\backslash \partial M$ contributes an interval $[
(r_{min}^{ext}, \text{Ord}), (R^+, \text{Rel}))$ to $S^L_0$. However there is no death at $(R^+, \text{Rel})$ in $S^{\partial{M}}_0$. This implies that $r_{min}^{ext}\notin \im(\hat{\phi}_0)$. This implies that $\hat{\phi}_0$ is a bijection from $\mathfrak{b}_0(\partial M,\rho_c)$ to $\mathfrak{b}\backslash \{(r_{min}^{ext}, \text{Ord})\}$.

As $c\in M$, $[(0,\text{Ord}),(r_{\max}, \text{Rel}))$ is an interval in the interval decomposition of $\text{XRPH}_0(M,c)$. Furthermore, observe that $[(r_{\min}^\mathrm{ext},\text{Ord}),(r_{\max}, \text{Rel}))$ is an interval in the interval decomposition of $\text{XRPH}_0(\partial M,c)$. 
Hence, we have 
$$\hat{\phi}_0(r_{\min}^\mathrm{ext},\text{Ord})=(0,\text{Ord}),$$ and 
$$\hat{\phi}_0\big(\mathfrak{b}_0(\partial M,\rho_c)\backslash \{(r_{\min}^\mathrm{ext},\text{Ord})\}\big) = \mathfrak{b}\backslash \{(0,\text{Ord}), (r_{\min}^\mathrm{ext},\text{Ord})\} = \mathfrak{b}_0(\partial M,\rho_c)\backslash \{(r_{\min}^\mathrm{ext},\text{Ord})\}.$$

It then follows $\hat{\phi}_0\vert_{\mathfrak{b}_0(\partial M,\rho_c)\backslash\{(r_{\min}^\mathrm{ext},\text{Ord})\}}$ is the identity by Lemma \ref{lem:identity}.

As we already know the matching of the remaining birth and death elements in $\mathfrak{b}$ and $\mathfrak{d}$, we have shown that \begin{align}\label{eq:k=0cimM}\text{XRPH}_0(\partial M,c)\oplus \mathcal{I}_{[(0,\text{Ord}),(r_{\max},\text{Rel}))}\oplus \mathcal{I}_{[(r_{\min}^\mathrm{ext},\text{Ord}),(R^+,\text{Rel}))}&= \\ \text{XRPH}_0(M,c)\oplus \text{XRPH}_0(L,c)&\oplus \mathcal{I}_{[(r_{\min}^\mathrm{ext},\text{Ord}),(r_{\max},\text{Rel}))}. \end{align}

We can then identify which of the intervals in $S_{0}^{\partial M}\backslash \{[(r_{\min}^\mathrm{ext},\text{Ord}),(r_{\max},\text{Rel}))\}$ are in $S_{0}^M$ (as opposed to $S_{0}^L$) by checking the sign of the birth values, using the same arguments as in the proof of Theorem \ref{thm:k}.
\end{proof}

Let $\mathfrak{b}_0^{\Ord}(X,f)$, and $\mathfrak{b}_0^{\Ess}(X,f)$ be the restriction of birth times to those intervals in the ordinary part $\Ord_0(X,\rho_c)$, and essential part $\Ess_0(X,\rho_c)$, of the extended persistence module $XRPH(X, \rho_c)$.

\begin{lemma}\label{lem:Ord0births}
$$\mathfrak{b}_0^{\Ess}(\partial M,\rho_c) \sqcup \{(0, \Ord)\}=\mathfrak{b}_0^{\Ess}(M,\rho_c)\sqcup \mathfrak{b}_0^{\Ess}(L,\rho_c)$$ and 
$$\mathfrak{b}_0^{\Ord}(\partial M,\rho_c)=\mathfrak{b}_0^{\Ord}(M,\rho_c)\sqcup \mathfrak{b}_0^{\Ord}(L,\rho_c).$$
\end{lemma}
\begin{proof}
We first prove $\mathfrak{b}_0^{\Ess}(\partial M,\rho_c) \sqcup \{(0, \Ord)\}=\mathfrak{b}_0^{\Ess}(M,\rho_c)\sqcup \mathfrak{b}_0^{\Ess}(L,\rho_c)$ by induction on the number of connected components of $\partial M$.

Suppose that $\partial M$ contains a single connected component. Then both $M$ and $L$ each consist of a single connected component. One of these contains $c$, and hence has its corresponding essential class with birth $(0, \Ord)$. The other, which doesn't contain $c$, with have minimum $\rho_c$ value on $\min\{\rho_c(\partial M)\}$ and hence its corresponding essential class with birth $(\min\{\rho_c(\partial M)\}, \Ord)$. Together these imply that $\mathfrak{b}_0^{\Ess}(\partial M,\rho_c) \sqcup \{(0, \Ord)\}=\mathfrak{b}_0^{\Ess}(M,\rho_c)\sqcup \mathfrak{b}_0^{\Ess}(L,\rho_c)$. 

For the inductive step, suppose that $\partial M$ has $m>1$ connected components. As $M$ is connected this implies that there is a connected component $L'$ of $L$ whose boundary $\partial L'$ is connected. If we set $\tilde{M}:=M\cup L'$ then $\partial \tilde{M}=\partial M \backslash \partial L'$, with $m-1$ connected components, and $\overline(B(\hat{c}, R)\backslash M)=L\backslash L'$. The inductive hypothesis tells us  $$\mathfrak{b}_0^{\Ess}(\partial \tilde{M},\rho_c) \sqcup \{(0, \Ord)\}=\mathfrak{b}_0^{\Ess}(\tilde{M},\rho_c)\sqcup \mathfrak{b}_0^{\Ess}(L\backslash L',\rho_c).$$ 
If $c\notin L'$ then $\mathfrak{b}_0^{\Ess}(L',\rho_c)=\mathfrak{b}_0^{\Ess}(\partial L',\rho_c)$ and $\mathfrak{b}_0^{\Ess}(\tilde{M},\rho_c)=\mathfrak{b}_0^{\Ess}(M,\rho_c)$. Together, with the inductive hypothesis, they tell us
\begin{align*}
    \mathfrak{b}_0^{\Ess}(\partial M,\rho_c) \sqcup \{(0, \Ord)\}&=\mathfrak{b}_0^{\Ess}(\partial \tilde{M},\rho_c) \sqcup \{(0, \Ord)\}\sqcup \mathfrak{b}_0^{\Ess}(\partial L',\rho_c)\\
    &=\mathfrak{b}_0^{\Ess}(\tilde{M},\rho_c)\sqcup \mathfrak{b}_0^{\Ess}(L\backslash L',\rho_c)\sqcup \mathfrak{b}_0^{\Ess}(L',\rho_c)\\
    &= \mathfrak{b}_0^{\Ess}(M,\rho_c)\sqcup \mathfrak{b}_0^{\Ess}(L,\rho_c).
\end{align*}
If $c\in L'$ then  $\mathfrak{b}_0^{\Ess}(L',\rho_c)=\{(O, \Ord)\}=\mathfrak{b}_0^{\Ess}(\tilde{M}, \rho_c)$ and $\mathfrak{b}_0^{\Ess}(M, \rho_c)=\mathfrak{b}_0^{\Ess}(\partial L',\rho_c)$. Together, with the inductive hypothesis, they tell us
\begin{align*}
    \mathfrak{b}_0^{\Ess}(\partial M,\rho_c) \sqcup \{(0, \Ord)\}&=\mathfrak{b}_0^{\Ess}(\partial \tilde{M},\rho_c) \sqcup \{(0, \Ord)\}\sqcup \mathfrak{b}_0^{\Ess}(\partial L',\rho_c)\\
    &=\mathfrak{b}_0^{\Ess}(\tilde{M},\rho_c)\sqcup \mathfrak{b}_0^{\Ess}(L\backslash L',\rho_c)\sqcup \mathfrak{b}_0^{\Ess}(\partial L',\rho_c)\\
    &= \mathfrak{b}_0^{\Ess}(L',\rho_c)\sqcup \mathfrak{b}_0^{\Ess}(L\backslash L',\rho_c)\sqcup \mathfrak{b}_0^{\Ess}(M, \rho_c)\\
    &=\mathfrak{b}_0^{\Ess}(L, \rho_c) \sqcup \mathfrak{b}_0^{\Ess}(M, \rho_c).
\end{align*}
As $\partial M$ always has finitely many connected components we have shown by induction that $\mathfrak{b}_0^{\Ess}(\partial M,\rho_c) \sqcup \{(0, \Ord)\}=\mathfrak{b}_0^{\Ess}(M,\rho_c)\sqcup \mathfrak{b}_0^{\Ess}(L,\rho_c)$ always holds.
 
For all sets $\Rel_0(X, \rho_c)=0$, and thus $\birth_0(X, \rho_c)=\birth_0^{\Ord}(X, \rho_c) \sqcup \birth_0^{\Ess}(X, \rho_c)$ for each of $X\in \{\partial M, M, L\}$.

From Lemma \ref{lem:births0} we have
$\mathfrak{b}_0(M,\rho_c)\sqcup  \mathfrak{b}_0(L,\rho_c)= \mathfrak{b}_0(\partial M,\rho_c)\sqcup \{(0,\text{Ord})\}$. The result then follows because we have already shown that $\mathfrak{b}_0^{\Ess}(\partial M,\rho_c) \sqcup \{(0, \Ord)\}=\mathfrak{b}_0^{\Ess}(M,\rho_c)\sqcup \mathfrak{b}_0^{\Ess}(L,\rho_c).$

\end{proof}

\begin{theorem}
Assume that $c\notin M$. If $$\text{XRPH}_0(\partial M,c) = \bigoplus_{[b_i,d_i)\in {S_0^{\partial M}}} \mathcal{I}_{[b_i,d_i)}$$ and $J_0^M\subseteq S_0^{\partial M}$ be the subset of intervals $[b_i,d_i)$ such that  either $b_i = (\rho_c(p),\text{Ord})$ with $p\in\text{Crit}(\rho_c^M,(0,+1))$,
then $$\text{XRPH}_0(M,c) = \mathcal{I}_{[(r_{min}, \text{Ord}), (r_{max}, \text{Rel}))}\oplus \bigoplus_{[b_i,d_i)\in J^M_0} \mathcal{I}_{[b_i,d_i)}.$$
\end{theorem}

\begin{proof}

We should first recall that we are assuming throughout that $M$ is connected. This implies that $\mathrm{Ess}_0(M,\rho_c)$ contains a single interval module born at $(\min_{p\in M}\rho_c(p), \mathrm{Ord})=(r_{min}, \mathrm{Ord})$ and dying at $(\max_{p\in M}\rho_s(p), \mathrm{Rel})=(r_{max},\mathrm{Rel})$. As always, $\mathrm{Rel}_0(M, \rho_c)$ must be empty. It remains to determine $\mathrm{Ord}_0(M)$.

By the same arguments as the first part of the proof of Theorem \ref{thm:0cinM} there exists an injective map 
$$\hat{\phi}_0: \birth_0(\partial M, \rho_c) \to \birth_0(M, \rho_c)\cup \birth_0(L, \rho_c)$$
that matches the birth of an interval $[b,d)$ in the interval decomposition of $\text{XRPH}_0(\partial M, c)$ to the birth of the interval $[b',d)$ (same death value) in the interval decomposition of $\text{XRPH}_0(M,c)\oplus \text{XRPH}_0(L,c)$ such that $b'\leq b$.

This injective map $\hat{\phi}_0$ will automatically match intervals within $\mathrm{Ord}_0(\partial M, \rho_c)$ to intervals in $\mathrm{Ord}_0(M, \rho_c)\oplus \mathrm{Ord}_0(L, \rho_c)$.
That is, $$\hat{\phi}_0|_{\birth_0^{\Ord}(\partial M, c)}: \birth_0^{\Ord}(\partial M, c) \to \birth_0^{\Ord}(M, c)\sqcup \birth_0^{\Ord}(L, c).$$

By Lemma \ref{lem:Ord0births} we have $\mathfrak{b}_0^{\Ord}(\partial M,\rho_c)=\mathfrak{b}_0^{\Ord}(M,\rho_c)\sqcup \mathfrak{b}_0^{\Ord}(L,\rho_c).$ As $\mathfrak{b}_0^{\Ord}$ is finite, 
we further know $\hat{\phi}_0|_{\mathfrak{b}_0^{\Ord}}$ is the identity by Lemma \ref{lem:identity}.

This implies we can compute $\mathrm{Ord}_0(M,\rho_c)\oplus \mathrm{Ord}_0(L,\rho_c)=\mathrm{Ord}_0(\partial M,\rho_c)$.

We can identify which of the intervals module summands in $\mathrm{Ord}_0(M,\rho_c)\oplus \mathrm{Ord}_0(L,\rho_c)$ are in $\mathrm{Ord}_0(M,\rho_c)$ by checking the sign of the birth values, as discussed in the proof of Theorem \ref{thm:k}.
\end{proof}

We will now handle the case of dimension $n-1$. The case where $r_{min}=r_{min}^{ext}$, when $c$ is closest to the  exterior component of $\partial M$, has a much more straightforward proof. While it could be treated purely as a special case of the general theorem it is pedagogically useful to cover it first.  

\begin{theorem}\label{r_min=r_min^ext} 
Assume $r_{\min}^\mathrm{ext}=r_{\min}$. 
Let $$\text{XRPH}_{n-1}(\partial M, c) = \bigoplus_{[b_i,d_i)\in S_{n-1}^{\partial M}} \mathcal{I}_{[b_i,d_i)}$$ and $J_{n-1}^M\subseteq S_{n-1}^{\partial M}$ be the subset of intervals $[b_i,d_i)$ such that $b_i = (\rho_c(p),\text{Ord})$ with $p\in\text{Crit}(\rho_c^M,(n-1,+1))$ or $b_i = (\rho_c(p),\text{Rel})$ with $p\in\text{Crit}(\rho_c^M,(n-(n-1)-1,-1))=\text{Crit}(\rho_c^M,(0,-1))$. 

Then $$\text{XRPH}_{n-1}(M, c) = \bigoplus_{[b_i,d_i)\in J_{n-1}^{M}} \mathcal{I}_{[b_i,d_i)}.$$
\end{theorem}

\begin{proof}

From LES (\ref{LESord}) and (\ref{LESrel}), we have that the maps
$$H_{n-1}((\partial M)_s)\to H_{n-1}(M_s)\oplus H_{n-1}(L_s),$$ and 
$$H_{n-1}(\partial M,(\partial M)^s)\to  H_{n-1}(M,M^s)\oplus H_{n-1}(L,L^s)$$ are surjective. So the induced morphism by inclusion
$$\phi_{n-1}: \text{XRPH}_{n-1}(\partial M,c) \to\text{XRPH}_{n-1}(M,c) \oplus \text{XRPH}_{n-1}(L,c)$$ is surjective.

From Lemma \ref{lem:births0} we know $\mathfrak{b}_{n-1}(M,\rho_c)\sqcup  \mathfrak{b}_{n-1}(L,\rho_c)= \mathfrak{b}_{n-1}(\partial M,\rho_c)$ so 
the induced matching theorem in \cite{matching} then tells us that there exists a bijection 
$$\hat{\phi}_{n-1}: \death_{n-1}(\partial M, \rho_c) \to \death_{n-1}(M, \rho_c)  \sqcup \death_{n-1}(L, \rho_c)$$ 
such that if $[b,d)\in S_{n-1}^{\partial M}$, then we have
$[b,\hat{\phi}_{n-1}(d))\in S_{n-1}^{M}\cup S_{n-1}^{L}$ with $\hat{\phi}_{n-1}(d)\leq d$.

Since $M$ is connected, we observe that $[(r_{\max},\text{Ord}),(R^-,\text{Rel}))\in S_{n-1}^{L}$ which corresponds to the exterior boundary component of $M$. Note that $[(r_{\max},\text{Ord}),(r_{\min}^\mathrm{ext},\text{Rel}))\in S_{n-1}^{\partial M}$; the one that corresponds to the exterior boundary component of $M$. Hence, we have $$\hat{\phi}_{n-1}((r_{\min}^\mathrm{ext},\text{Rel}))=(R^-,\text{Rel}).$$

From Lemma \ref{lem:births0}, $\mathfrak{d}_{n-1}(\partial M,\rho_c)\backslash \{(r_{\min},\text{Rel})\} =\mathfrak{d}_{n-1}(M,\rho_c)\sqcup  \mathfrak{d}_{n-1}(L,\rho_c)\backslash \{(R^-,\text{Rel})\}$. Thus $\hat{\phi}_{n-1}|_{\mathfrak{d}_{n-1}(\partial M,\rho_c)\backslash \{(r_{\min},\text{Rel})\}}$ is a bijection of a finite set to itself. By Lemma \ref{lem:identity} it must be the identity.

We have shown that 
\begin{align}\label{eq:k=n-1cimM}
\text{XRPH}_{n-1}(\partial M,c)&\oplus \mathcal{I}_{[(r_{\max},\text{Ord}),(R^-,\text{Rel}))}\\ \nonumber
&= \text{XRPH}_{n-1}(M,c)\oplus \text{XRPH}_{n-1}(L,c)\oplus \mathcal{I}_{[(r_{\max},\text{Ord}),(r_{\min}^\mathrm{ext},\text{Rel}))}. 
\end{align}

We can then identify which of the intervals in $S_{n-1}^{\partial M}\backslash \{[(r_{\max},\text{Ord}),(r_{\min}^\mathrm{ext},\text{Rel}))\}$ are in $S_{n-1}^M$ (as opposed to $S_{n-1}^L$) by checking the sign of the birth values, using the same arguments as in the proof of Theorem \ref{thm:k}.

\end{proof}

Suppose now $r_{\min}^\mathrm{ext}\neq r_{\min}$. Note that we still have a bijection 
$$\hat{\phi}_{n-1}:\death_{n-1}(\partial M, \rho_c) \to \death_{n-1}(M, \rho_c)\sqcup \death_{n-1}(L, \rho_c)$$ as defined in the proof of Theorem \ref{r_min=r_min^ext}. 
Moreover, we still have that $[(r_{\max},\mathrm{Ord}),(R^-,\mathrm{Rel))}\in S_{n-1}^{L}$ and $[(r_{\max},\mathrm{Ord}),(r_{\min}^{\mathrm{ext}},\mathrm{Rel))}\in S_{n-1}^{\partial M}$, so it follows that  $\hat{\phi}_{n-1}((r^\mathrm{ext}_{\min}, \mathrm{Rel}))=(R^-,\mathrm{Rel})$. However, since $r_{\min}^\mathrm{ext}\ne r_{\min}$, we still need to determine $\hat{\phi}_{n-1}^{-1}((r^\mathrm{ext}_{\min}, \mathrm{Rel}))$.

Consider the essential classes in $Ess_{n-1}(\partial M, \rho_c))$, which can be formulated in terms of the connected components of $\partial M$.The connected components consist of one exterior boundary component $A^{\ext}$ (the unique  connected component in the infinite component in $\R^n \backslash M$) and an indexed set $\{A_i\}$ of interior connected components.  Let $r(b_i)=\max_{x\in A_i}\rho_c(x)$, and $r(d_i)=\min_{x\in A_i} \rho_c(x)$. Each $[A_i]$ corresponds to an essential $n-1$ class of $\partial M$ with birth parameter $(r(b_i), \mathrm{Ord})$ and death parameter $(r(d_i), \mathrm{Rel})$. That is, $Ess_{n-1}(\partial M, \rho_c))$ has intervals $[(r(b_i), \textrm{Ord}), (r(d_i), \textrm{Rel}))$.

Let $j_1$ be the index such that $r(d_{j_1}) < r^\mathrm{ext}_{\min}$ and for all $i\neq j_1$, either $r(b_i)<r(b_{j_1})$ 
and/or $r(d_i)>r^\mathrm{ext}_{\min}$. We claim that $[(r(b_{j_1}), \mathrm{Ord}), (r_{\min}^\mathrm{ext}, \mathrm{Rel}))\in S_{n-1}^{M}$. 

Consider the $(n-1)$-chain 
$$\sigma=\sum_{\{k\,:\,r(b_k)\leq r(b_{j_1})\}}[A_k].$$ 
As we only include components that are entirely contained in $M_{r(b_{j_1})}$, we have $\sigma\in H_{n-1}(M_{r(b_{j_1})})$. Since the coefficent of $[A_{j_1}]$ is non-zero, we have that $\sigma$ is in the cokernel of the transition map $H_{n-1}(M_{r(b_{j_1})-\epsilon}, \emptyset)\to H_{n-1}(M_{r(b_{j_1})}, \emptyset)$. That is, $\sigma$ is born at $(r(b_{j_1}),\mathrm{Ord})$.

Note that for $t>r_{\min}^{\ext}$ we have $H_{n-1}(M, M^t)$ has a basis $\{[A_k]\,:\, r(d_k)<t\}$.

At $(r^\mathrm{ext}_{\min}, \mathrm{Rel})$, the $(n-1)$-cycle $\sum_i [A_i]$ becomes a relative boundary. This is because $\sum_i [A_i]$ differs from $[A^\mathrm{ext}]$ by the $n$-chain $M$, and we have $[A^\mathrm{ext}]$ becomes a relative boundary at $(r^\mathrm{ext}_{\min}, \mathrm{Rel})$. 

Now $\sum_k [A_k] - \sigma=\sum_{\{i\,:r(b_i) > r(b_{j_1})\}} [A_i]$. By construction if $r(b_i)> r(b_{j_1})$ then $r(d_i)>r_{\min}^{\ext}$. This implies that $\sum_k [A_k]-\sigma$ is a relative boundary in $H_{n-1}(M, M^{r_{\min}^{\ext}})$ as $r(d_i) > r^\mathrm{ext}_{\min}$ implies $A_i\subset M^{r^\mathrm{ext}_{\min}}$. 

From these arguments we can deduce that $[(r(b_{j_1}), \mathrm{Ord}), (r_{\min}^\mathrm{ext}, \mathrm{Rel}))\in S_{n-1}^{M}$. Combined with $[(r_{\max},\mathrm{Ord}),(R^-,\mathrm{Rel))} \in S_{n-1}^{L}$ we know  $\hat{\phi}_{n-1}((r(b_{j_1}),\mathrm{Rel})=(r_{\min}^\mathrm{ext},\mathrm{Rel})$. If $r(j_1)=r_{\min}$ then we could proceed as in the proof of Theorem \ref{r_min=r_min^ext} to deduce $\text{XRPH}_{n-1}(M, c)$ from $\text{XRPH}_{n-1}(\partial M, c)$. If $r(j_1)\neq r_{\min}$ then we would still need to determine $\hat{\phi}_{n-1}((r(j_1), \textrm{Rel}))$. This process creates a cascading effect through the deaths of the essential classes.

\begin{defn}\label{defn:cascade}
Let $\psi_{\partial M}: \mathfrak{d}_{n-1}(Ess_{n-1}(\partial M), \rho_c) \to \mathfrak{b}_{n-1}(Ess_{n-1}(\partial M), \rho_c)$ match each death with its corresponding birth.
For each death $d\in \mathfrak{d}_{n-1}(Ess_{n-1}(\partial M), \rho_c)$ we can consider the set of deaths that occur after it $\{\hat{d}\in \mathfrak{d}_{n-1}(Ess_{n-1}(\partial M), \rho_c)\mid \hat{d}>d\}$. This set will contain $(r_{\min}, \textrm{Rel})$ for all $d\neq (r_{\min}, \textrm{Rel})$ and hence will never be empty. Define the cascading function \begin{align*}
\Cascade: \mathfrak{d}_{n-1}(Ess_{n-1}(\partial M), \rho_c) \backslash \{(r_{\min}, \text{Rel})\}  & \to 
\mathfrak{d}_{n-1}(Ess_{n-1}(\partial M), \rho_c)\\  
d&\mapsto \argmax_{\{\hat{d}\mid \hat{d}> d\}} \psi_{\partial M}(\hat{d}).
\end{align*} That is, the death of the youngest essential class that dies after $d$.
\end{defn}

We will be interested in a sequence of death values that are iterations of this cascade function, starting with $(r_{\min}^{\text{ext}}, \text{Rel})$. 

\begin{defn}\label{def:cascadeSeq}
The \emph{death cascade sequence} of $M$ is the sequence $d_{j_1}, d_{j_2}, \ldots d_{j_K}$ with $d_{j_1}=(r_{\min}^{\text{ext}}, \text{Rel})$, $d_{j_K}=(r_{\min}, \text{Rel})$ and $d_{j_{l+1}}=\Cascade(d_{j_{l}})$ for $l<K$.   
\end{defn}

Note that if $r_{\min}= r_{\min}^{\text{ext}}$ then the death cascade sequence is a single element. From now on we will focus on death cascade sequences with at least two elements. 

\begin{lemma}\label{lem:decreaseradii}
The death radii, and the birth radii, decrease along the death cascade sequence. That is $$r(d_{j_1})>r(d_{j_2})>\ldots r(d_{j_K})$$ 
and 
$$r(b_{j_1})>r(b_{j_2})>\ldots r(b_{j_K}).$$
\end{lemma}
\begin{proof}
As $d_{j_{l+1}}\in \{d\mid d> d_{j_l}\}$ we automatically have $(r(d_{j_l}), \Rel)= d_{j_{l}}>d_{j_{l-1}}=(r(d_{j_{l}}), \Rel)$. Thus $r(d_{j_l})<r(d_{j_{l-1}})$. 

By construction $r_{\max}$ is the last birth in the $Ess_{n-1}(\partial M, \rho_c)$ so $r(b_{j_1})>r(b_{j_l})$ for all $l\neq 1$. 

Let $1<l<K$. Both $d_{j_l}$ and $d_{j_{l+1}}$ are in the set $\{d \mid d\geq d_{j_{l-1}}\}$. By construction $\psi_{\partial M}(d_{j_l})>\psi_{\partial M}(d_{j_{l+1}})$ and hence $r(b_{j_l})>r(b_{j_{l+1}})$.
\end{proof}

\begin{lemma}\label{lem:orderbirths}
Consider an interior connected component $A_i$ of $\partial M$, with birth $(r(b_i), \textrm{Ord})$ and death $(r(d_i), \textrm{Rel})$, and $d_{j_l}$ in the death cascade. 
\begin{itemize}
\item If $r(b_i) > r(b_{j_{l}})$ then $r(d_i)>r(d_{j_{l+1}})$. 
\item If $r(d_i)< r(d_{j_{l+1}})$ then $r(b_i) < r(b_{j_l})$ 
\end{itemize}
\end{lemma}
\begin{proof}
If $r(d_i) < r(d_{j_{l+1}})$ then $d_i>d_{j_{l}}$. By construction $\psi_{\partial M}(d_{j_{l+1}})>\psi_{\partial M}(d_{i})$ and hence $r(b_i)<r(b_{j_{l+1}})$. The second statement is the contrapositive statement, using that all critical values are distinct.
\end{proof}

It will be useful to have notation for the set of indices that in the death cascade sequence that correspond to interior connected components of $\partial M$. These are 
$\{j_1, \ldots j_{K}\}$ and we denote by the set by $\DC$. 

\begin{lemma}\label{lem:Ess_n-1(L)}
$Ess_{n-1}(L,\rho_c)=\mathcal{I}_{[(r_{\max}, Ord), (R^-,Rel))}$. 
\end{lemma}

\begin{proof}
Consider the connected components of $L$. There is exactly one that is in the infinite component of $\R^n\backslash M$. Call this component $L^{\text{ext}}$. Observe that the boundary of $L^{\text{ext}}$ is the exterior boundary component $A^{\text{ext}}$ of $\partial M$ union a sphere of radius $R$, that is $\partial (L^{\text{ext}})=A^{\text{ext}}+ \partial B(\hat{c}, R)$ (considered as a chain over $\mathbb{Z}_2$). In $H_{n-1}(L, L^s)$, $\partial B(\hat{c},R)$ becomes a relative boundary precisely when $s=R^-$. This implies that $A^{\text{ext}}=\partial (L^{\text{ext}})-\partial B(\hat{c}, R)$ also becomes a relative boundary. This implies that $\mathrm{XRPH}_{n-1}(L, c)$ has a essential class corresponding to the interval $[(r_{\max}, \mathrm{Ord}), (R^-, \mathrm{Rel}))$.

Let $L'$ be a different connected component of $L$, and consider it as a subset of $S^n$, the one-point compactification of $\R^n$.  Since $M$ is connected and $L'$ is not in the infinite component of $\R^n\backslash M$ we also know that $S^n\backslash L'$ is connected. By Alexander duality $H_{n-1}(L')= \tilde{H}^0(S^n \backslash L')=0$. 
This implies that $L'$ contributes no essential persistent homology classes to $\mathrm{XRPH}_{n-1}(L, c)$. The result follows as $\mathrm{XRPH}_{n-1}(L,c)$ is the direct sum of the extended radial persistent homology of its connected components.
\end{proof}

For each $i\notin \DC$ let $$\sigma_i:=[A_i].$$ 
For each $j_l\in \DC$ let $$\sigma_{j_l}:=\sum_{\{k\,\mid \, r(b_k)\leq r(b_{j_l})\}}[A_k].$$

\begin{lemma}\label{lem:basis}
Fix $t>0$. Then $\{\sigma_i \mid r(b_i)\leq t\}$ is a basis for $H_{n-1}(M_t)$.
\end{lemma}
\begin{proof}
Geometrically, it is clear that $\{[A_i] \mid \max\{\rho_c(A_i)\}\leq t\}$ is a basis for $H_{n-1}(M_t)$. By construction the map from $\{[A_i] \mid \max\{\rho_c(A_i)\}\leq t\}$ to $\{\sigma_i \mid r(b_i)\leq t\}$ is a change of basis.
\end{proof}

We also need to investigate bases for the relative homology $H_{n-1}(M, M^t)$. For this, it will be convenient to split into the two cases of when $t>r_{\min}^{\ext}$ and $t\leq r_{\min}^{\ext}$.

\begin{lemma}\label{lem:relativebasisbig}
Fix $t>r_{\min}^{\ext}$. Then $$\{\sigma_i \mid r(d_i)< t\}$$ is a basis for $H_{n-1}(M, M^t)$.
\end{lemma}    
\begin{proof}
We have $\{[A_i] \mid r(d_i)< t\}$ is a basis for $H_{n-1}(M, M^t)$. Every element of $H_{n-1}(M, M^t)$ can be uniquely written as a linear combination of $\{[A_i] \mid r(d_i)< t\}$.

As $t>r_{\min}^{\ext}$ we also have $t>r(d_{j_l})$ for all $j_l\in \DC$ by Lemma \ref{lem:decreaseradii}. For each $j_l\in \DC$ we have $\sigma_{j_l}$ is homologous to $\sum_{\{k\mid r(b_k)\leq r(b_{j_l}) \text{ and } r(d_k)< t\}}[A_k]$.

If we order the $\{[A_i] \mid r(d_i)< t\}$ by increasing birth times we can write $\{\sigma_i \mid r(d_i)< t\}$ as a matrix with each column its linear combination in terms of $\{[A_i] \mid r(d_i)< t\}$. As this matrix is triangular with non-zero diagonal entries we see $\{\sigma_i \mid r(d_i)< t\}$ must also be a basis for $H_{n-1}(M, M^t)$.
\end{proof}

\begin{lemma}\label{lem:relativebasissmall}
Assume $r_{\min} \neq r_{\min}^{\ext}$ and let $t\in (r_{\min}, r_{\min}^{\ext}]$. Then $$\{\sigma_i \mid i\notin \DC \text{ and } r(d_i)< t\}\cup \{\sigma_{j_l} \mid {j_l}\in \DC \text{ and } r(d_{j_{l-1}})<t\} $$ is a basis for $H_{n-1}(M, M^t)$.
\end{lemma}  
\begin{proof}
As $r_{\min}< t \leq r_{\min}^{\ext}$ the relative boundaries (up to boundaries in $M$) are generated by $\{[A_i] \mid \min\{\rho_c(A_i)\}\geq t\}$ alongside $\sum_k [A_k]$ (the sum over all interior connection components of $\partial M$). 

Let $P=m_1, m_2, \ldots m_p$ be the indices of the interior connected components of $\partial M$ such that $\min \{\rho_c(A_{m_j})\}<t$. Order them by birth in $\mathrm{XRPH}_{n-1}(\partial M)$, that is $\max\{\rho_c(A_{m_1})\}<\max\{\rho_c(A_{m_2})\}\ldots <\max\{\rho_c(A_{m_p})\}$.

Note that since $r_{\min}<t$ there must be at least one $m_j$ in the death cascade index set (in particular, $i$ such that $\min\{\rho_c(A_i)\}=r_{\min}$ has $i\in P$ and $i\in \DC$). 

Let $m_q$ be the largest $m_j$ such that $m_q\in \DC$. By Lemma \ref{lem:decreaseradii} we have $d_{m_q}$ is also the earliest in the death cascade sequence restricted to those with $r(d_{j_l})< t$.

The set $\{[A_{m_1}], [A_{m_2}], \ldots , [A_{m_{q-1}}], [A_{m_{q+1}}], \ldots [A_{m_p}]\}$ is a basis for $H_{n-1}(M, M^t)$ (any choice of $p-1$ of them would form a basis). We can write each of the $\sigma_i$, considered as  a relative homology class in $H_{n-1}(M, M^t)$, uniquely as a linear combination of these $$\{[A_{m_1}], [A_{m_2}], \ldots , [A_{m_{q-1}}], [A_{m_{q+1}}], \ldots [A_{m_p}]\}.$$

Consider $i\notin DC$. If $\min \{\rho_c(A_i)\}\geq t$ then $\sigma_i$ is a relative boundary and $0$ when written as a linear combination of $\{[A_{m_1}], [A_{m_2}], \ldots , [A_{m_{q-1}}], [A_{m_{q+1}}], \ldots [A_{m_p}]\}$. If $\min \{\rho_c(A_i)\}< t$ then $i=m_j$ for some $j\neq q$ and $\sigma_i$ can be written as $[A_i]$.

Consider $j_l \in \DC$ and suppose $\min \{\rho_c(A_{j_{l-1}})\} \geq t$. 
By construction we have $\sum_k [A_k]-\sigma_{j_l}=\sum_{\{k\mid r(b_k)>r(b_{j_l})\}}[A_k]$. 
By Lemma \ref{lem:orderbirths} if $r(b_i)>r(b_{j_l})$ then $r(d_i)>r(d_{j_{l-1}})$. 
That is $\min\{\rho_c(A_i)\}\geq \min \{\rho_c(A_{j_{l-1}})\} \geq t$ and $[A_i]$ is a relative boundary. This implies that $\sum_k [A_k]-\sigma_{j_l}$, and hence also $\sigma_{j_l}$, is a relative boundary in $H_{n-1}(M, M^t)$.

Finally consider $j_l \in \DC$ with $\min \{\rho_c(A_{j_{l-1}})\} < t$. As $\min \{\rho_c(A_{j_{l}})\} < \min \{\rho_c(A_{j_{l-1}})\}$ by construction we have both $j_l$ and $j_{l-1}$ in $P$. As $j_{l-1}$ appears later in the order of the indices in $P$ we have $j_l\neq m_q$.  
We have $\sigma_i=\sum_{\{k\mid r(b_k)<r(b_{j_l})\}}[A_k]$ is equivalent in $H_{n-1}(M, M^t)$ to $\sum_{\{k \mid r(b_k)<r(b_{j_l}) \text{ and } k\in P\}}[A_k]$.

\end{proof}

To complete the picture we can also observe that $H_{n-1}(M, M^t)=0$ for $t\leq r_{\min}$.

\begin{cor}
For $i\notin \DC$ we have $\sigma_i$ is not a relative boundary in $H_n(M, M^t)$ for all $t>r(d_i)$. For $j_l\in \DC$, $l>1$, we have  $\sigma_{j_l}$ is not a relative boundary in $H_n(M, M^t)$ for all $t>r(d_{j_{l-1}})$.
\end{cor}

We can combine the above results to read off the $Ess_{n-1}(M,\rho_c)$.

\begin{prop}\label{prop:ess_n-1}
Let $d_{j_1}, d_{j_2}, \ldots d_{j_K}$ be the death cascade sequence.
$$Ess_{n-1}(M,\rho_c)=\bigoplus_{i\notin \DC} \mathcal{I}_{[b_i, d_i)} \oplus \bigoplus_{l=2}^{K}\mathcal{I}_{[b_{j_l}, d_{j_{l-1}})}.$$
\end{prop}
\begin{proof}
We have explicitly found a basis $\{\sigma_k\}$ for this interval decomposition. This set of elements, when non-zero for a choice of $t$, provide a basis $H_{n-1}(M_t)$ and $H_{n-1}(M, M^t)$, as shown by combining Lemmas \ref{lem:basis}, \ref{lem:relativebasisbig} and \ref{lem:relativebasissmall}.
\end{proof}

\begin{theorem}\label{dim}
Let $$\text{XRPH}_{n-1}(\partial M, c) = \bigoplus_{[b_i,d_i)\in S_{\partial M}} \mathcal{I}_{[b_i,d_i)}.$$ 
Let $d_{j_1}, d_{j_2}, \ldots d_{j_K}$ be the death cascade sequence of $\partial M$.
Let $J_{n-1}^M\subseteq S_{n-1}^{\partial M}$ be the subset of intervals $[b_i,d_i)$ such that $b_i = (\rho_c(p),\text{Rel})$ with $p\in\text{Crit}(\rho_c^M,(0,-1))$.

Then 
$$\text{XRPH}_{n-1}(M, c)=\bigoplus_{i\notin \DC} \mathcal{I}_{[b_i, d_i)} \oplus \bigoplus_{l=2}^{K}\mathcal{I}_{[b_{j_l}, d_{j_{l-1}})}\oplus \bigoplus_{[b_i,d_i)\in J_{n-1}^M}\mathcal{I}_{[b_i, d_i)}.$$
\end{theorem}
\begin{proof}
Lemma \ref{lem:births0} says
$\mathfrak{d}_{n-1}(M,\rho_c)\sqcup  \mathfrak{d}_{n-1}(L,\rho_c)= \big(\mathfrak{d}_{n-1}(\partial M,\rho_c)\sqcup \{(R^-,\text{Rel})\} \big)\backslash \{(r_{\min},\text{Rel})\}$. $\partial M$ is an $(n-1)$ dimensional manifold, if has no deaths in the ordinary parameter range, so neither does $M$. Hence $\Ord_{n-1}(M, \rho_c)=0$. 

Recall that $$\text{XRPH}_{n-1}(M, c)=\Ess_{n-1}(M, \rho_c)\oplus \Ord_{n-1}(M, \rho_c) \oplus \Rel_{n-1}(M, \rho_c).$$ From Proposition \ref{prop:ess_n-1} we have $Ess_{n-1}(M,\rho_c)=\bigoplus_{i\notin \DC} \mathcal{I}_{[b_i, d_i)} \oplus \bigoplus_{l=2}^{K}\mathcal{I}_{[b_{j_l}, d_{j_{l-1}})}.$ Combined with $\Ord_{n-1}(M, \rho_c)=0$,
it is thus sufficient to show that $\text{Rel}_{n-1}(M, c)=\bigoplus_{[b_i,d_i)\in J_{n-1}^M}\mathcal{I}_{[b_i, d_i)}.$

The first part of the proof is the same as that of Theorem \ref{r_min=r_min^ext}. We have a bijection $$\hat{\phi}_{n-1}: \death_{n-1}(\partial M, \rho_c) \to \death_{n-1}(M, \rho_c)\cup \death_{n-1}(L, \rho_c)$$ 
such that if $[b,d)\in S_{n-1}^{\partial M}$, then we have
$[b,\hat{\phi}_{n-1}(d))\in S_{n-1}^{M}\cup S_{n-1}^{L}$ with $\hat{\phi}_{n-1}(d)\leq d$.

As a death corresponds to an essential class if and only if it is of the form $(\min\{\rho_c(A_i)\},\Rel)$ for some connected component $A_i$, we have $\hat{\phi}_{n-1}$ restricts to a bijection $\hat{\phi}_{n-1}|_{\death_{n-1}(\Rel_{n-1}(\partial M, \rho_c))}: \death_{n-1}(\Rel_{n-1}(\partial M, \rho_c)) \to \death_{n-1}(\Rel_{n-1}(M, \rho_c))\cup \death_{n-1}(\Rel_{n-1}(L, \rho_c))$.
Since $$\death_{n-1}(\Rel_{n-1}(\partial M, \rho_c))=\death_{n-1}(\Rel_{n-1}(M, \rho_c))\cup \death_{n-1}(\Rel_{n-1}(L, \rho_c))$$ we can apply Lemma \ref{lem:identity} to conclude $\hat{\phi}_{n-1}|_{\death_{n-1}(\Rel_{n-1}(\partial M, \rho_c))}$ is the identity. Furthermore, this implies $\Rel_{n-1}(\partial M,\rho_c)=\Rel_{n-1}(M, \rho_c)\oplus \Rel_{n-1}(L,\rho_c)$. 

We can then use the method of checking signs in the proof of Theorem \ref{thm:k} to identify the interval modules in $\Rel_{n-1}(\partial M,\rho_c)$ that are also in $\Rel_{n-1}(M, \rho_c)$. 
\end{proof}

\section{Application to binary images}
\label{application}

\subsection{Representation of binary digital images}
\label{consider}

Binary images are $m\times n$ matrices or 2-dimensional arrays, with entries called \emph{pixels} taking exactly one of two values, 0 (black) or 1 (white). We call the structure we are interested in computing persistent homology for the \emph{foreground}.  This will consist of either all the black pixels or all the white pixels.  

To implement radial filtration in binary images, we modified the algorithm described in \cite{KT2022}. The original algorithm\footnote{The corresponding R package is available at \url{https://github.com/james-e-morgan/xpht}.} 
is implemented in R and takes a binary image as input and outputs the extended persistent homology transform of the foreground with respect to an even number of directions. 
We adjust it to compute radial extended persistent homology for a single radial filtration, i.e. by fixing a center point. Figure \ref{JM} displays the workflow of the algorithm. In what follows, we will discuss some necessary theoretical justifications and highlight some practical considerations. Justifications for other parts of the algorithm can be found in Section 6 of \cite{KT2022}. 

\begin{figure}[h]
    \centering
    \begin{tikzpicture}[
    node distance = 2mm and 4mm,
      start chain = going right,
 disc/.style = {shape=cylinder, draw, shape aspect=0.3,
                shape border rotate=90,
                text width=15mm, align=center, font=\linespread{0.8}\selectfont},
  mdl/.style = {shape=ellipse, aspect=2.1, draw},
  alg/.style = {draw, align=center, font=\linespread{0.7}\selectfont}
                    ]
    \begin{scope}[every node/.append style={on chain, join=by -Stealth}]
\node (n1) [disc] {Binary \\ image $M$};
\node (n2) [alg]  {Extract\\ foreground\\ boundary \\ curves};
\node (n3) [alg]  {Compute XRPH of each \\ boundary curve for a \\given centre point and \\ find critical point signs};
\node (n4) [alg] {Decide the XRPH \\classes for $M$};
\node (n3) [disc]  {Output\\XRPH \\ of $M$};
    \end{scope}
    \end{tikzpicture}
    \caption{Workflow diagram to compute XRPH with respect to a given radial filtration on the foreground of an input binary image.}
    \label{JM}
\end{figure}
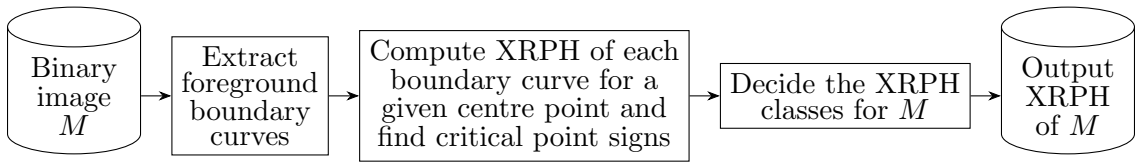

As the results of Section~\ref{sec:proof} establish, to find $\mathrm{XRPH}(M)$ it suffices to compute $\mathrm{XRPH}(\partial M)$. There are many methods to add structure to pixel arrays to make them suitable for topological analysis. 
A standard approach treats a pixel array as a cubical complex, with each pixel treated as either a $0$-dimensional element (vertex) in this cell complex or alternatively, as a top-dimensional element or $n$-dimensional cube~\cite{DualDigitalImage}.   

In the current application we wish to extract the boundary of a set of pixels and need each boundary component to be piecewise-linear 1-manifold. As discussed in~\cite{KT2022}, we do this by first treating each pixel as a unit square patch centered at integer coordinates $(i,j)$, and adopt the convention that foreground pixels connect through corner adjacencies. We then construct vertex points for the boundary curves by placing a vertex in the middle of an edge whenever the two pixels on either side take distinct values. These boundary vertices are connected with oriented line segments that have the foreground on the left, see Figure~\ref{midpts}.

\begin{figure}[ht]
    \centering
    \includegraphics[width=0.5\textwidth]{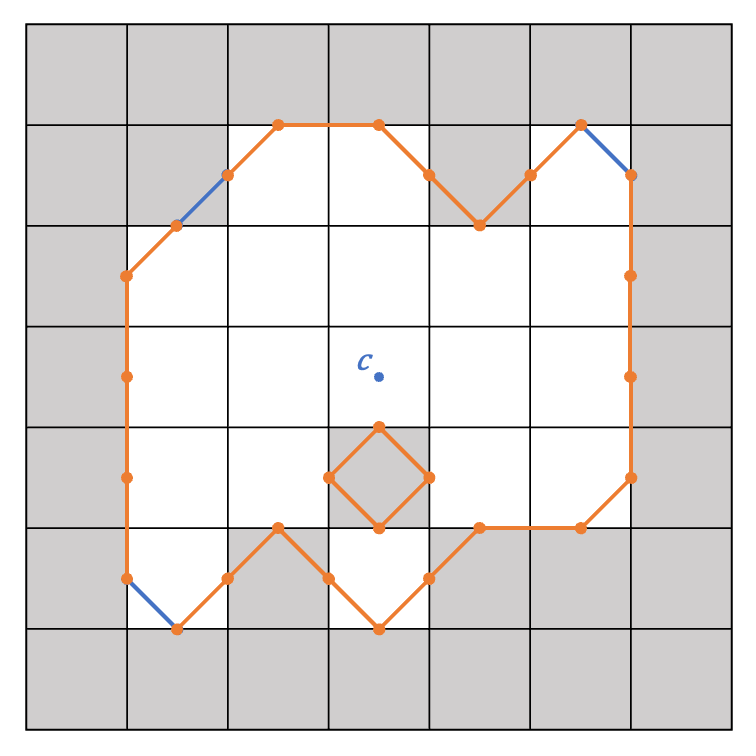}
    \caption{An example that illustrates how we construct the boundary curves from a binary image. Pixels are square cells such that the foreground is white (`1') and the background is gray (`0'). 
    The foreground pixels connect through corner adjacencies as seen at the bottom of the image. 
    Boundary line segments are determined by the pixel values in each $2\times 2$ patch. There are $2^4$ possible configurations for these patches and at least one of each type is illustrated in the above example. 
    If we choose $c$ to be at the center of a pixel, as shown above, then the blue edges have both their vertex values at the same distance from $c$. For this reason, we always perturb the coordinates of $c$ so that they do not take integer or half-integer values, and also exclude the cases where $\text{frac}(c_x) = \pm\text{frac}(c_y)$. 
    }
    \label{midpts}
\end{figure}

\subsection{Sign detection for critical points}
\label{signs}
Now suppose for a given binary image input $M$, we have traced its boundary $\partial M$ by recording the sequence of boundary vertices. Fix a point $c$, and consider the radial function $\rho_c$. Since the intersection of a circle with a straight line is at most two points, unlike in \cite{KT2022}, we do not need to consider the case where more than three boundary vertices on the same straight line take the same value (the ``co-linear" case in \cite{KT2022}). 
However, it is possible for two adjacent boundary vertices to have the same distance from $c$, and for this reason, we perturb the coordinates of $c$ so that they do not take integer or half-integer values, and also exclude the case where $\text{frac}(c_x) = \pm\text{frac}(c_y)$. 
We then make the following definition:

\begin{defn}\label{localmin}
Let $\gamma$ be a piecewise-linear curve in the boundary $\partial M$ of $M$ with $m$ vertices ordered cyclically $x_0, x_1,\cdots,x_m=x_0.$ Then a vertex $x_i$ is a \emph{0-critical point} if $\rho_c(x_{i-1})>\rho_c(x_i)$ and $\rho_c(x_i)<\rho_c(x_{i+1})$, where the subscripts are taken modulo $m$. We say $x_i$ is a \emph{local minimum} or alternatively, a \emph{$(+)$-critical point} if there exists some $\epsilon>0$ such that for all $a\in B(x_i,\epsilon)\cap M$, we have $\rho_c(x_i)\leq \rho_c(a).$
\end{defn}

The main change to the original algorithm to now work with radial filtrations occurs in the process of labelling $0$-critical points. 
Recall that in the first step, we find the set of boundary curves and then the algorithm computes
the radial distance from $c$ for the vertices on the boundary curves. 
These values determine when a vertex is a homological critical point for the boundary curve. 
It is the sign of the critical point that allows us to decide if it corresponds to a birth or death event for $M$. 
For example, when considering $0$-critical points for $\partial M$, the $(+)$-critical points are related to Ordinary classes, while the others, $(-)$-critcal points, are related to Relative classes. This was justified in Section \ref{sec:proof}.

To differentiate between the two types of critical point in practice, we use the following fact about orientation of triangles:

\begin{lemma}\label{triangle}
Let $\Delta XYZ$ be a triangle with positive area. Denote by $\det(x,y)$, 
the determinant of a $2\times 2$ matrix whose columns are $x$ and $y$. 
Then the vertices of $\Delta XYZ$ are in an anticlockwise order if $\det(Z-Y,X-Y)>0$, and they are in a clockwise order if $\det(Z-Y,X-Y)<0$.
\end{lemma}

Now we use Lemma \ref{triangle} to find the sign of the $0$-critical points.

\begin{theorem}\label{localMinTest}
Let $M\subseteq\mathbb{R}^2$ be a compact set whose boundary is a disjoint union of piecewise-linear closed curves. Let $\gamma$ be a piecewise-linear curve of $\partial M$ with vertices $x_0,x_1,\cdots,x_m=x_0$ traversed counterclockwise with respect to $M$.  
Fix a point $c\in\mathbb{R}^2$ and consider the radial function $\rho_c$ centered at $c$. 
A 0-critical point of $\gamma$, $x_i$ is a $(+)$-critical point for $\rho_c$ if and only if $\det(x_{i-1}-c,x_i-c)<0.$
\end{theorem}

\begin{proof}
Let $x_i$ be a 0-critical point. Consider the triangle $\Delta x_icx_{i-1}$. Then every point $p$ on the interior of the edge $x_ic$ satisfies $\rho_c(p)<\rho_c(x_i)$.

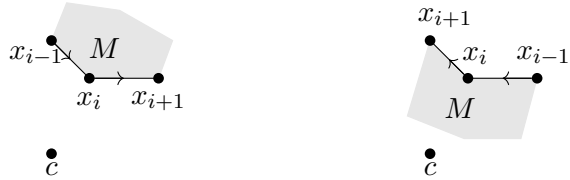
\begin{figure}[h]
    \centering
    \begin{tikzpicture}
    \node[circle, fill = black, inner sep=0pt,minimum size=4pt] (c1) at (0,0) {};
    \node (c1b) at (0,-0.2) {$c$};
    \draw[fill = gray!20, gray!20] (0,1.5) -- (0.5,1) -- (1.414,1) -- (1.6,1.5) -- (0.9,1.9)--(0.2,2) -- (0,1.5);
    \draw (0,1.5) -- (0.5,1) -- (1.414,1);
    \node[circle, fill = black, inner sep=0pt,minimum size=4pt] (xi-1) at (0,1.5) {};
    \node at (-0.2,1.3) {$x_{i-1}$};
    \node[circle, fill = black, inner sep=0pt,minimum size=4pt] (xi) at (0.5,1) {};
    \node at (0.5,0.7) {$x_{i}$};
    \node[circle, fill = black, inner sep=0pt,minimum size=4pt] at (1.414,1) {};
    \node at (1.414,0.7) {$x_{i+1}$};
    \draw[->] (xi-1) -- (0.25,1.25);
    \draw[->] (xi) -- (0.957,1);
    \node at (0.7,1.4) {$M$};
    \node[circle, fill = black, inner sep=0pt,minimum size=4pt] (c1) at (5,0) {};
    \node (c1b) at (5,-0.2) {$c$};
    \draw[fill = gray!20, gray!20] (5,1.5) -- (5.5,1) -- (6.414,1) -- (6.2,0.2) -- (5.45,0.2) -- (4.7,0.5) -- (5,1.5);
    \node at (5.4,0.6) {$M$};
    \draw (5,1.5) -- (5.5,1) -- (6.414,1);
    \node[circle, fill = black, inner sep=0pt,minimum size=4pt] at (5,1.5) {};
    \node at (5.2,1.8) {$x_{i+1}$};
    \node[circle, fill = black, inner sep=0pt,minimum size=4pt] (xi2) at (5.5,1) {};
    \node at (5.6,1.3) {$x_i$};
    \node[circle, fill = black, inner sep=0pt,minimum size=4pt] (xii) at (6.414,1) {};
    \node at (6.5,1.3) {$x_{i-1}$};
    \draw[->] (xii) -- (5.957,1);
    \draw[->] (xi2) -- (5.25,1.25);
    \end{tikzpicture}
    \caption{An illustration of the two cases discussed in the proof: the first case (top left) where $x_i$ is a $(+)$-critical point and the second case (top right) where $x_i$ is not. The shaded area indicates a small portion of $M$.}
    \label{criticalPoint}
\end{figure}

Suppose $x_i$ is a $(+)$-critical point for $\rho_c$. Then there exists some $\epsilon>0$ such that for all $a\in B(x_i,\epsilon)\cap M$, we have $\rho_c(a)\geq \rho_c(x_i)$. It follows that $p\not\in B(x_i,\epsilon)\cap M$ for all $p$ on $x_ic$. Hence, $\Delta x_icx_{i-1}$ does not contain a subset of $B(x_i,\epsilon)\cap M$, which means $c$ is on the opposite side of $x_ix_{i-1}$ compared to $M$. Since $\gamma$ is traversed in a counterclockwise order, we must have $x_i,c,x_{i-1}$ in a clockwise order. Therefore, by Lemma \ref{triangle}, we have $\det(x_{i-1}-c,x_i-c)<0$ as required.

Conversely, suppose $x_i$ is not a $(+)$-critical point for $\rho_c$. Then the triangle $\Delta x_icx_{i-1}$ contains a subset of $B(x_i,\epsilon)\cap M$, i.e. $c$ is on the same side of $x_ix_{i-1}$ as $M$. Again, since $\gamma$ is traversed counterclockwise, we must now have $x_i,c,x_{i-1}$ in a counterclockwise order, which by Lemma \ref{triangle}, we know that $\det(x_{i-1}-c,x_i-c)>0$.
\end{proof}

\subsection{Practical considerations}
\label{practical}

In practice, for a given binary image, our algorithm represents the boundary of the foreground by a list of points traced counterclockwise with respect to the foreground, as depicted by the orange dots and line segments in Figure \ref{midpts}. 
For a radial filtration, there will be places where the true sublevel set filtration includes part of an edge before its endpoints; see the blue edges in Figure~\ref{midpts} for example. 
Note that for any choice of center $c$, however, there will be values of $r$ for which the circle of radius $r$ about $c$ has a point of tangency with an edge of $\partial M$. 
In practice, we build the sublevel set filtration in a simplex-wise manner and include an edge $x_{i-1} x_i$ only when both $\rho_c(x_{i-1}) \leq r$ and $\rho_c(x_{i}) \leq r$.
 
Stability results justify this algorithmic choice. The error from the approximation is bounded by the pixel spacing.  
To see this, let $\rho_c$ denote the true radial function on the piecewise linear boundary curve $\partial M$ and $\widehat{\rho_c}$ the function induced on $\partial M$ by the simplex-wise filtration that assigns the larger vertex value to the whole edge. 
For each $s>0$, if $y\in (\partial M)_s$, then $\rho_c(y)=\lVert y- c\rVert_2\le s$. 
Let  $x_{i-1}$ and $x_i$ be the endpoints of the edge containing $y$. 
Then $\hat{\rho}_c(y)=\max\{\rho_c(x_{i-1}),\rho_c(x_i)\}.$ Since $\max\{\lVert y-x_{i-1}\rVert_2,\lVert y-x_i\rVert_2\}\le 1$ (assuming integer-spaced pixels) we have $$\hat{\rho}_c(y)=\max\{\lVert x_i-c\rVert_2, \lVert x_{i-1}-c\rVert_2\}\le \max\{\lVert y-x_{i-1}\rVert_2,\lVert y-x_i\rVert_2\}+\lVert y-c\rVert_2\le \rho_c(y)+1.$$ Therefore, we have $\lvert\hat{\rho}_c(y)-\rho_c(y)\rvert\le 1$, and the extended persistence modules  
obtained from $\rho_c$ and $\hat{\rho}_c$ are $1$-interleaved, leading to a bounded bottleneck distance. 
Moreover this approximation error is the same as the discretization of the pixel array, the limiting accuracy in feature detection from the image.    

Finally, the theory developed in Section \ref{sec:proof} assumes that $\rho_c$ is Morse and that all critical values are distinct. This may not be the case for digital images. For example, if the center $c$ is chosen to be the center of a pixel, then $\rho_c$ can take the same value on many boundary vertices due to their regular spacing. 
However, the set of centers $c$ such that $\rho_c$ is \emph{not} Morse for a given image $M$ has measure 0.
That is, for all $c$, there is a generic point $c'$ close enough to $c$ such that the signs of critical points on the boundary of the image do not change, and that $\rho_{c'}$ is Morse. Stability results (e.g., in \cite{C2016}) then tell us that we can obtain the XRPH$(M,c)$ from that of XRPH$(M,c')$.

\section{Future directions}

In this paper, we proposed a novel efficient computing of extended radial persistence matching for a manifold with boundary by examining its relationship with the extended persistence matching for the boundary of the given manifold. We also provided implementation in R that computes the extended radial persistence matching for image data with minimum critical value achieved on the exterior boundary of the foreground. In this case, the cascade function gives a simple death cascade sequence for each connected component of the image. Note that this covers most practical scenarios. For instance, if the image of interest has no hole, then any choice of centre will automatically lead to a radial function satisfying the assumption $r_{\min}=r_{\min}^{\mathrm{ext}}$. Similarly, if the centre is chosen in the infinite component of the foreground of the image, then the corresponding radial function will also satisfy the assumption. That said, it would still be a useful exercise to have code implementation of the cascade algorithm so we can compute extended radial persistence matching of the foreground for all choices of centre for the radial filtration. 

In Section \ref{practical}, we saw that the current algorithmic implementation computes the XRPH of the foreground of an image using the approximate radial function $\hat{\rho}_c$ on its boundary. However, since the foreground boundary is piecewise linear, the true radial function on the boundary cell complex will almost always be Morse. The critical points will be where the circle centred at $c$ first becomes tangent to the boundary lines. The corresponding radii are the critical values. So if we add in tangent points of the circles centred at $c$ to the boundary curve as vertices and record their distance to the centre $c$, we can compute the true XRPH of the boundary cell complex. The justification we gave in Section \ref{practical} demonstrated that our algorithm computing the XRPH of the radial filtration using the approximate radial function $\hat{\rho}_c$ still gives the charaterisation of the XRPH for the true radial filtration computed using $\rho_c$. Nevertheless, it might be of interest to implement the alternative method to compute the XRPH for the actual boundary cell complex. 

Another direction is to explore whether a nice correspondence between regular persistence and extended persistence exists for other types of filtration. In particular, it would be interesting to see whether the method of proof exploited here could be extended to alternative settings. 

Finally, we look forward to seeing this theoretical advancement being used in real-life applications.

\section{Appendix}
\subsection{R code}
We provide R codes for implementing our algorithm on binary images for the case where $r_{\min}=r_{\min}^{\mathrm{ext}}$. The location of $c$ can be either inside $M$ or outside $M$. Assumptions made for interpreting binary images are included in Section \ref{consider}. The code can be found at \url{https://github.com/JencyJ12/XRPH.git}.

\subsection{Pseudocode}\label{code}
We include below pseudocode for how we implement the results from Section \ref{sec:proof} in the case of binary images. So in what follows, $M$ will be the foreground of some binary image. It can be regarded as a manifold with boundary embedded in $\R^2$. 

Following the workflow in Figure \ref{JM}, we would start by extracting the boundary $\partial M$ of $M$. We use a standard square tracing algorithm for this, except in our case we also want to consider the orientation of the curve to differentiate between interior and exterior boundaries. In particular, we will record the numbers of exterior and interior boundaries $(m,n)$, the coordinates of exterior boundaries and interior boundaries in the output list {\bf bdryCurve}. 

As we did in Section \ref{sec:proof}, we can break down the foreground into disjoint connected components, and then the extended persistence matching of the foreground of some given input binary image is the direct some of that for each connected components of the foreground. Hence, in the following algorithms, we will describe what we do to each connected components of the foreground before combining them into a matching for the complete binary image. In particular, that means that the numbers of exterior and interior boundaries for each connected componnet will be of the form $(1,n)$ instead.

While four different cases (whether the centre is chosen inside $M$ and whether $r_{\min}=r_{\min}^{\mathrm{ext}}$) for computing the extended radial  persistence modules are considered depending on the position of the centre point relative to the foreground manifold that we are interested in, most parts of both algorithms are the same -- it is only during the translation from extended persistence of the boundary to extended persistence of the foreground that they start to differ. In light of this, we will break down the algorithm into small code blocks so that it is clearer to both the author and audience of this paper the structure of the algorithm.
We start by translating the boundary components we obtained from tracing the image foreground into a structure recording the vertices, edges and coordinates of the boundaries of each connected components as in Algorithm \ref{alg:parseSkeleton}. We assume that exterior and interior boundary components (\textbf{extBndry} and \textbf{intBndry}) are $n_i\times 2$ matrices with rows containing coordinates of pixel midpoints on the corresponding boundary in order so that adjacent rows corresponds to neighboring vertices on the boundary. This translates the given boundary curves into a 1-skeleton or equivalently, gives a complete description of the 2-skeleton of the corresponding connected component. Since it is possible that multiple vertices on the same connected component are of the same distance away from the centre, such a labeling also allows us to use the one with lower index to be observed first scanning outwards for consistency and symmetry breaking. 

\begin{algorithm}[H]
\caption{\textbf{parseSkeleton} Translate given boundary curves into skeleton structure.}\label{alg:parseSkeleton}
\begin{algorithmic}
\State \textbf{Input:} $\boldsymbol{n}$, the number of interior boundary components;

    \;\;\;\;\;{\bf extBndry}, a matrix whose rows are coordinates of the vertices on the  exterior boundary of the given connected component;
    
    \;\;\;\;\;{\bf intBndry$[i]$}, a matrix whose rows are coordinates of the vertices on the $(i-1)$-th interior boundary components for $2\leq i\leq n+1$.

\State \textbf{Output: {\bf skeleton}}, a list of length $n+1$ with the following structure ($i=1$ corresponds to the exterior boundary component): 
\begin{itemize}
    \item {\bf skeleton$[[i]]$.vertex}, a vector numbering the $n_i$ vertices of the corresponding boundary component;
    \item {\bf skeleton$[[i]]$.edges}, a sequence of vectors of the form $(j,j+1)$ (mod $n_i$) recording the vertices of ordered edges;
    \item {\bf skeleton$[[i]]$.coords}, a sequence of coordinates on the $i$-th boundary curve. 
\end{itemize}
\algstore{bkbreak} \end{algorithmic} \end{algorithm}

\begin{algorithm}[H]
\begin{algorithmic}
\algrestore{bkbreak}
\For{\textbf{all }$i\in \{1,2\cdots,n+1\}$}
    \If{$i=1$}
        \State curve$_i \gets $ \textbf{extBndry} \Comment{Note that {\bf skeleton[[1]]} corresponds to the exterior boundary component.}
    \Else
        \State curve$_i \gets $ \textbf{intBndry}$[i-1]$
    \EndIf
    
    \If{the first row and the final row of curve$_i$ coincides} 
    
    \Comment{Check if the $i$-th curve is a loop}
    \State np$_i$ $\gets $ number of rows in curve$_i$ $- 1$
    \Else
    \State np$_i$ $\gets $ number of rows in curve$_i$
    \EndIf
    
    \State $\textbf{skeleton$[[i]]$.vertex} \gets 1:\;np_i$
    
    \For{\textbf{all} $j\in\{1,2,\cdots,\text{np}_{i}-1\}$}
        \State \textbf{skeleton$[[i]]$.edges}.add($(j,j+1)$)
        \State \textbf{skeleton$[[i]]$.coords}.add(coordinates of the $j$-th vertex on curve$_i$)
    \EndFor
    \State \textbf{skeleton$[[i]]$.edges}.add($(\text{np$_i$},1)$)
    \State \textbf{skeleton$[[i]]$.coords}.add(coordinates of the final vertex on curve$_i$)
\EndFor
\end{algorithmic}
\end{algorithm}

The next step is to compute the radial filtration of a connected component of the foreground given some centre point, in such a way that allows us to compute the XRPH of the boundary. To that end, we will record for each vertex whether its neighbors are closer to the centre (so the filtration sees the neighbor(s) before it sees the vertex). We demonstrate how we achieve it \emph{for a single boundary component} in Algorithm \ref{alg:computeRadialFiltration}. 

\begin{algorithm}[H]
\caption{\textbf{computeRadialFiltration.} Compute radial filtration for the given skeleton for some fixed centre point.}\label{alg:computeRadialFiltration}
\begin{algorithmic}
\State \textbf{Inputs: curve}, an object taking the same form of \textbf{skeleton$[[i]]$} from Algorithm \ref{alg:parseSkeleton};

\;\;\;\;\;\;\;\,\textbf{centre}, a coordinate vector of the centre point.
\State \textbf{Output: filtration}, a list of length 3 with the following structure:
\begin{itemize}
    \item \textbf{filtration.radius}, a vector of distances from each vertex on the curve to the centre
    \item \textbf{filtration.lowerNbrs}, a list for each vertex recording its neighbours closer to the centre
    \item \textbf{filtration.coords}, a sequence of coordinates of vertices on the curve
\end{itemize}

\vspace{0.2cm}
\State \textbf{filtration.coords} $\gets$ \textbf{curve.coords}
\For{\textbf{each} \textbf{vertex} in \textbf{curve.vertex}}
\State \textbf{filtration.radius}.add(distance between \textbf{vertex} and \textbf{centre})
\EndFor
\algstore{bkbreak} \end{algorithmic} \end{algorithm}

\begin{algorithm}[H]
\begin{algorithmic}
\algrestore{bkbreak}
\For{\textbf{all} i in $\{1,2,\cdots,\text{length(\textbf{curve.vertex})}\}$}
\State e$_i \gets$ \textbf{curve.edges}$[i]$ \Comment{Edge of \textbf{curve} that starts at $i$-th vertex}
\State $r_i \gets$ \textbf{filtration.radius}$[e_i[1]] $\Comment{Radius of the $i$-th vertex}
\State $r_{i+1} \gets$ \textbf{filtration.radius}$[e_i[2]]$\Comment{Radius of the next vertex on edge $e_i$}
\If{$r_i<r_{i+1}$}
    \State \textbf{filtration.lowerNbrs}$[[e_i[2]]]$.add($e_i[1]$)
\ElsIf{$r_i=r_{i+1}$}
    \If{$e_i[1]<e_i[2]$} \Comment{The one with lower index is a lower neighbour of the one with higher index}
    \State \textbf{filtration.lowerNbrs}$[[e_i[2]]]$.add$(e_i[1])$
    \Else
    \State \textbf{filtration.lowerNbrs}$[[e_i[1]]]$.add($e_i[2]$)
    \EndIf
\Else
    \State \textbf{filtration.lowerNbrs}$[[e_i[1]]]$.add($e_i[2]$)
\EndIf
\EndFor
\end{algorithmic}
\end{algorithm}

With Algorithm \ref{alg:computeRadialFiltration}, we are now able to compute XRPH \emph{for a single boundary component}. Here we also introduce a threshold parameter (\textbf{tolerance} in Algorithm \ref{alg:computeRPH}) that allows us to ignore potential noise in the foreground. Moreover, on each boundary component, the only essential class has interval representation $\mathcal{I}_{[r,R)}$, where $r$ and $R$ are the smallest and largest distances from the centre achieved on this boundary component respectively.

\begin{algorithm}[H]
\caption{\textbf{computeBndryXRPH.} Compute XRPH of the boundary of a connected component}\label{alg:computeRPH}
\begin{algorithmic}
\State {\bf Input: filtration}, output from Algorithm \ref{alg:computeRadialFiltration};
\State \;\;\;\;\;\;\;\;\;\;\;\; \textbf{tolerance}, numeric parameter that controls the noise from the input, i.e. if the distance between two vertices is less than this value, then any class generated by them will be ignored;
\State \;\;\;\;\;\;\;\;\;\;\;\; \textbf{centre}, a coordinate vector of the centre point.
\State {\bf Output: finite}, components of the boundary with finite survival time, i.e. corresponding to Ord and Rel classes;
\State \;\;\;\;\;\;\;\;\;\;\;\;\;\;\;\; \textbf{essential}, components of the boundary with infinite survival time, i.e. corresponding to Ess classes.
\vspace{0.2cm}
\State {\bf distList} $\gets$ sorted distinct values in {\bf filtration.radius} from small to large
\If{tail({\bf distList})$-$head({\bf distList}$)>$\textbf{tolerance}}\Comment{Check if the component is considered noise}
    \State \textbf{essential} $\gets$ (head({\bf distList}),tail({\bf distList}))
\EndIf
\algstore{bkbreak} \end{algorithmic} \end{algorithm}

\begin{algorithm}[H]
\begin{algorithmic}
\algrestore{bkbreak}
\For{$r$ in {\bf distList}} 
    \State $v_r$ $\gets$ vertices with radius $r$
    \For{$v$ in $v_r$}
        \If{$v$ has no lower neighbour in \textbf{filtration.lowerNbrs$[[v]]$}}
            \State $p_v\gets v$
        \Else
            \State \textbf{component}$\gets$ unique roots of paths to lower neighbours of $v$
        
            \If{\textbf{component} has exactly one element} 
                \State $p_v\gets $ \textbf{component}
            \Else
                \State \textbf{bTimes} $\gets$ distances from entries in \textbf{component} to \textbf{centre}
                \State \textbf{minBTime} $\gets \min(\textbf{bTimes})$ 
                \State reorder \textbf{component} from indices small to large

                \For{$x$ in \textbf{component}}
                    \State $r_x\gets$ \textbf{filtration.radius}$[[x]]$
                    \If{$\textbf{minBTime}<r_x<r$ and $r-r_x>\textbf{tolerance}$}
                        \State \textbf{finite}.add($(r_x,r)$) 
                    \ElsIf{$r_x=\textbf{minBTime}$ and $p_v$ has not been assigned}
                        \State \textbf{newComponent}$\gets x$
                        \State $p_v\gets \textbf{newComponent}$
                    \ElsIf{$\textbf{minBTime}=r_x < r$ and $r-r_x>\textbf{tolerance}$}
                        \State \textbf{finite}.add($(r_x,r)$)
                    \EndIf
                \EndFor
                \For{x in \textbf{component}}
                    \State $p_x\gets\textbf{newComponent}$\Comment{All components found are part of the same connected component.}
                \EndFor
            \EndIf
        \EndIf
    \EndFor
\EndFor
\end{algorithmic}
\end{algorithm}

Another step in mapping the XRPH of $\partial M$ to the XRPH of $M$ is to determine the sign of each critical point. We include pseudocode for computing the sign of a critical point based on Theorem \ref{localMinTest}.

\begin{algorithm}[H]
    \caption{\textbf{computeSign.}}\label{alg:testMin}
    \begin{algorithmic}
        \State \textbf{Input: $\boldsymbol{x}$}, $\boldsymbol{x}[1]$ =  vertex traced immediately before current critical point, 
        \State\;\;\;\;\;\;\;\;\;\;\;\;\;\;\;\;\;\;\;\;\;\;\;\;\;\;\;\;$\boldsymbol{x}[2]$ = current critical point;
        \State\;\;\;\;\;\;\;\;\;\;\;\; \textbf{c}, coordinate vector of the centre point;
        \State \textbf{Output: isPositive}, boolean, 1 = (+)-critical, 0 = $(-)$-critical

    \State $\boldsymbol{M}\gets \begin{bmatrix}
        \vert &  \vert\\
        \boldsymbol{x}[1]-\textbf{c} & \boldsymbol{x}[2]-\textbf{c}\\
        \vert &  \vert
    \end{bmatrix}$
    \State \textbf{isPositive} $\gets \det(\boldsymbol{M}) < 0$
    \end{algorithmic}
\end{algorithm}

We are now ready to translate XRPH of $\partial M$ to XRPH of each connected component of $M$. We will emphasize the use of Algorithm \ref{alg:testMin} with italics. Note that Algorithm \ref{alg:RXPH1} assumes that $M$ is connected. 

\begin{algorithm}[H]
\caption{\textbf{computeXRPH.} 
}\label{alg:RXPH1}
\begin{algorithmic} 
\State \textbf{Input: bdryCurve}, boundary components of a foreground connected component;
\State\;\;\;\;\;\;\;\;\;\;\;\; \textbf{centre}, coordinate vector of the centre point;
\State\;\;\;\;\;\;\;\;\;\;\;\; \textbf{tolerance}, numeric parameter that controls the noise from the input;
\algstore{bkbreak} \end{algorithmic} \end{algorithm}

\begin{algorithm}[H]
\begin{algorithmic}
\algrestore{bkbreak}
\State \textbf{Output: Ord0}, intervals in the Ordinary class of dimension 0;
\State\;\;\;\;\;\;\;\;\;\;\;\;\;\;\; \textbf{Rel1}, intervals in the Relative class of dimension 1;
\State\;\;\;\;\;\;\;\;\;\;\;\;\;\;\; \textbf{Ess0}, intervals in the Essential class of dimension 0;
\State\;\;\;\;\;\;\;\;\;\;\;\;\;\;\; \textbf{Ess1}, intervals in the Essential class of dimension 1.
\vspace{0.2cm}
\State \textbf{skeleton} $\gets$ \textbf{parseSkeleton}(\textbf{bdryCurve})

\For{each \textbf{curve} in \textbf{skeleton}}\Comment{Compute the RPH of the boundary of the connected component}
    \State \textbf{filtration} $\gets$ \textbf{computeRadialFiltration}(\textbf{curve},\textbf{centre})
    \State [\textbf{finite},\textbf{essential}].add(\textbf{computeBndryXRPH}(\textbf{filtration},\textbf{tolerance},\textbf{centre})
\EndFor
\For{\textbf{element} in \textbf{finite}}
    \If{birth location of \textbf{element} has a \textit{positive sign}}
        \State \textbf{Ord0}.add(\textbf{element})
    \Else
        \State \textbf{Rel1}.add(\textbf{element})
    \EndIf
\EndFor

\State $r_{\min}\gets$ minimum birth value in \textbf{essential}
\State $r_{\max}\gets$ maximum death value in \textbf{essential} \Comment{achieved on the exterior boundary component}

\If{vertex achieving $r_{\min}$ has a \textit{negative sign}} \Comment{\textbf{centre} is inside $M$} 
    \State \textbf{Ess0}.add((0,$r_{\max}))$
    \State \textbf{Rel2}.add(($r_{\min},0$))
\EndIf

\If{vertex achieving $r_{\min}$ has a \textit{positive sign}}\Comment{we are in the case $r_{\min} = r_{\min}^\mathrm{ext}$}
\For{\textbf{element} in \textbf{essential}}
    \If{birth location of \textbf{element} has a \textit{negative sign}}
        \State \textbf{Ess1}.add(\textbf{element})
    \EndIf
\EndFor
\Else \Comment{we have $r_{\min} \ne r_{\min}^\mathrm{ext}$} 
\State Needs the cascading algorithm
\EndIf
\end{algorithmic}
\end{algorithm}

\end{document}